\documentclass{amsart}
\usepackage{graphicx}
\usepackage{amssymb,amscd,amsthm,amsxtra}
\usepackage{latexsym}
\usepackage{epsfig}
\usepackage{mathtools}
\usepackage{esint}
\usepackage{color}

\newtheorem{thm}{Theorem}[section]
\newtheorem{cor}[thm]{Corollary}
\newtheorem{lem}[thm]{Lemma}
\newtheorem{prop}[thm]{Proposition}
\theoremstyle{definition}
\newtheorem{defn}[thm]{Definition}
\theoremstyle{remark}
\newtheorem{rem}[thm]{Remark}
\numberwithin{equation}{section}
\newcommand{\R}{\mathbb R}

\newcommand{\eps}{\epsilon}

\newcommand{\p}{\partial}
\DeclareMathOperator{\sgn}{sgn}

\begin{document}

\title[Minimizers of the double well energy]{Minimizers of the Allen-Cahn energy with sub-quadratic growth}
\author{Ovidiu Savin}
\address{Department of Mathematics, Columbia University, New York, NY 10027}\email{savin@math.columbia.edu}
\author{Chilin Zhang}
\address{School of Mathematical Sciences, Fudan University, Shanghai 200433, China}\email{zhangchilin@fudan.edu.cn}

\begin{abstract}
We establish Liouville theorems for global minimizers $u$ of the Allen-Cahn energy
$$\int |\nabla u|^2 + W(u) \, dx,$$
which have subquadratic growth at infinity. In particular we extend the results of \cite{S1,S3} concerning the De Giorgi's conjecture to the setting of unbounded solutions.
Part of the analysis relies on the regularity of minimizers for a Dirichlet/perimeter functional which was studied by Athanasopoulous-Caffarelli-Kenig-Salsa in \cite{ACKS}. 
\end{abstract}

\maketitle

\section{Introduction}

The typical energy functional associated with phase field models combines the Dirichlet integral of a density $u$ together with a potential term (or free energy) $W(u)$,
\begin{equation}\label{J}
J(u):=\int_\Omega \frac{ |\nabla u|^2}{2} + W(u) \, dx.
\end{equation}
In the classical example of the Allen-Cahn energy \cite{AlC}, the term $W$ is a double-well potential of the form 
\begin{equation}\label{J1}W(t)=(1-t^2)^2,
\end{equation} which is relevant in the theory of phase-transitions and minimal surfaces. In their celebrated result, Modica and Mortola \cite{MM} showed that $0$-homogenous rescalings of bounded minimizers $|u| \le 1$, converge up to subsequences to a $\pm 1$ configuration separated by a minimal surface, i.e.
\begin{equation*}\label{ueps}
u_\eps(x)= u\left( \frac x \eps \right)  \quad \to \quad \chi_E - \chi_{E^c}  \quad \mbox{in $L^1_{loc}$, \quad as $\eps \to 0$,}
\end{equation*}
with $E$ a set of minimal perimeter. At the level of the energy, this result is expressed in terms of the Gamma-convergence of the rescaled energies
to a multiple of the perimeter functional $c_0 Per_{\Omega}(E)$.

The connection between the level sets of $u$ at large scales, $\{u=t\}$ for a fixed $t \in (-1,1)$,  and minimal surfaces is a subject of great interest that was intensively studied in the literature, see for example \cite{M,St,HT, AC, GG, CM}.
A key difference between the two objects is that minimal surfaces remain invariant under dilations, while the level sets of $u$ do not. On the other hand, the level sets of $u$ are smooth for a.e. $t$ while minimal surfaces could have singularities or higher multiplicities. 

Some natural problems that arise in this context, known as De Giorgi type conjectures, ask whether or not global solutions of the two problems have the same rigidity properties. The original formulation of De Giorgi made in 1978 refers precisely to this question for bounded monotone solutions of the Allen Cahn equation 
\begin{equation}\label{AC}
\triangle u = W'(u),
\end{equation}
and the corresponding Bernstein problem for minimal graphs. In \cite{S1}, the first author proved the following version of the De Giorgi's conjecture concerning minimizers of the Allen-Cahn energy.

\begin{thm}[\cite{S1}]\label{DG}
Let $u$ be a global minimizer of \eqref{J}-\eqref{J1} with $|u| \le 1$. Then $u$ is one-dimensional if $n \le 7$.
\end{thm} 

The term global minimizer refers to a function $u$ defined in the whole space $\R^n$, that minimizes the energy in any ball subject to its own boundary conditions. The notion that $u$ is one-dimensional means that $u(x)=g(x \cdot \nu)$ with $\nu $ a unit direction and $g$ a function of one-variable which solves the ODE \eqref{AC}. 

The dimension $n=7$ turns out to be sharp just as in the case of area minimizing minimal surfaces, by a construction of Liu-Wang-Wei \cite{LWW}. A consequence of Theorem \ref{DG} is that it implies the graphical version of the De Giorgi conjecture up to dimension $n\le 8$ under the natural assumptions
$$ |u| \le 1, \quad u_{x_n}>0, \quad \quad \mbox{$\{u=0\}$ is a graph over $\R^{n-1}$ in the $x_n$ direction.}$$
We remark however that the original formulation of the De Giorgi conjecture was stated only for bounded monotone solutions of \eqref{AC} without the assumption that  $\{u=0\}$ is a graph over the whole $\R^{n-1}$. Under these weaker hypotheses the conjecture is known to be true only in dimensions 2 and 3 by the works of Ghoussoub-Gui \cite{GG}, Ambrosio-Cabr\`e \cite{AmC}, and to be false in dimension $n \ge 9$ by a counterexample due to Del Pino, Kowalczyk and Wei \cite{DKW}. Without the graphicality condition the rigidity question is closely related to the classification of global stable solutions in one dimension lower.

In this paper we extend the result of Theorem \ref{DG} to minimizers that are not necessarily bounded but that have subquadratic growth at infinity. Our main result is the following.

\begin{thm}\label{Main}
Let $u$ be a global minimizer of \eqref{J} with 
\begin{equation}\label{Wtr}
W(t)=
\left\{
\begin{array}{l}
(1-t^2)^2 , \quad \quad \mbox{if} \quad |t| \le 1,\\

\

\\

0, \quad \quad \quad \quad \quad  \mbox{if} \quad |t|>1,\\
\end{array}
\right .
\end{equation}
and assume that $u =o(|x|^2)$ as $|x| \to \infty$. Then $u$ is one-dimensional if $n \le 7$.
\end{thm}  

Notice that the potential $W$ in Theorem \ref{Main} is extended trivially outside the interval $[-1,1]$ and therefore the second term in the energy is relevant only in the range $|u|<1$. 

Similarly, Theorem \ref{Main} implies the graphical version of the De Giorgi conjecture with subquadratic growth in one dimension higher.

 \begin{thm}\label{Main2}
 Let $u$ be a global solution to \eqref{AC} with $W$ as in \eqref{Wtr}. If $u_{x_n}>0$, $u =o(|x|^2)$ as $|x| \to \infty$, and
$$ \mbox{$\{u=0\}$ is a graph over $\R^{n-1}$ in the $x_n$ direction,}$$
then $u$ is one-dimensional if $n \le 8$.
 \end{thm}

A few comments are in order regarding our main results. A first remark is that the Modica-Mortola Gamma-convergence result no longer applies under the hypotheses of Theorems \ref{Main} and \ref{Main2} since $u$ is not necessarily bounded. Now the blow-down analysis and the corresponding singular limit problems depend strongly on the behavior of $u$ at infinity. A heuristic computation of a possible singular limit problem in this setting was carried out by Athanasopoulous, Caffarelli, Kenig, Salsa in \cite{ACKS}. They considered minimizers of \eqref{J},\eqref{Wtr} which are not necessarily bounded, in order to motivate and introduce the Dirichlet/perimeter energy functional $\mathcal F$ (see \eqref{F} below) which they studied in \cite{ACKS}. The formal computation carried out in their paper suggests that if $ \max_{B_R}u \sim \sqrt R$ then the limiting problem for the blow-down rescalings $u(Rx)/ \sqrt R$ as we let $R \to \infty$, consists in minimizing an energy of the form
 \begin{equation}\label{F}
 \mathcal F(v):=\int_\Omega |\nabla v|^2 dx + Per_{\Omega}(\{ v> 0\}).
 \end{equation}

The regularity of minimizers of $\mathcal F$ was investigated in \cite{ACKS}. In particular the authors established the Lipschitz continuity of such minimizers and the smoothness of their free boundaries $\p \{v>0\}$ outside a set of Hausdorff dimension $n-8$. Not surprisingly, these results will play an important role in our analysis.

On the other hand it is not at all clear apriori how $u$ behaves at infinity. For example there are one-dimensional minimizers of $J$ which grow linearly at infinity, and possibly two dimensional minimizers which grow barely quadratically. The strategy to prove Theorem \ref{Main} is to show that either $|u| \le 1$, in which case the prior result Theorem \ref{DG} in \cite{S1} applies, or that eventually $u$ is asymptotically linear for a sequence of radii $R_k \to \infty$. Heuristically, if $u$ grows sufficiently fast at infinity then the solution at large scales behaves close to a harmonic function. The desired conclusion then follows from an improvement of flatness result which states that $u$ is better and better approximated by one-dimensional solutions as we restrict to smaller scales. 

Next we discuss some possible extensions of our results that we intend to explore in the future and a few related works. We expect the methods developed here to be relevant when considering potentials $W$ with power-like decay tails, like for example potentials of the form
$$W_\gamma(t)=(1+ t^2)^{-\gamma/2},\quad \quad \gamma>0,$$ 
which decay fast at infinity. Minimizers of $J_\gamma$ with $\gamma$ large, should have similar properties to minimizers of $J$ with potential $W$ as in \eqref{Wtr}. After appropriate homogenous rescalings, a limiting problem associated to $J_\gamma$ involves the minimization of the so called two-phase Alt-Phillips energy functional \cite{AP} for negative exponents 
$$\mathcal E_\gamma(v):= \int_\Omega |\nabla v|^2 + |v|^{-\gamma} \, dx.$$
The Alt-Phillips functional for negative exponents $\gamma \in (0,2)$ was investigated recently in \cite{DS1,DS2,DS3}. On the other hand, when $\gamma \ge 2$ the minimization problem for $\mathcal E_\gamma$ is ill posed for sign changing solutions since in this case there is an infinite amount of energy concentrating near the level set $\{v=0\}$. This indicates that the relevant limiting problems for phase-field models with power-decay tail potentials $W_\gamma$ experience a phase transition as the parameter $\gamma$ approaches the value 2: if $\gamma \ge 2$ then the Gamma convergence results involve Dirichlet/perimeter type energies like $\mathcal F$, while if $\gamma <2$ they involve energies of Alt-Phillips type $\mathcal E_\gamma$. These heuristics suggest that Theorem \ref{Main} could hold as well for global minimizers of $J_\gamma$ if $\gamma \ge 2$. For the remaining range of exponents $\gamma \in (0,2)$, the rigidity properties of global minimizers with subquadratic growth should inherit the same rigidity properties of global minimizers of the Alt-Phillips functional $\mathcal E_\gamma$. 

A similar family of potentials to consider are those that 
vanish to the left of $0$ and have the form
$$W_{0,\gamma}(t): = \varphi(t) (t^+)^{-\gamma},$$
where $\varphi$ is a smoothing of the characteristic function $\chi_{[0,\infty)}$, i.e. $\varphi$ is smooth, nondecreasing and $\varphi(t)=0$ if $t \le 0$ and $\varphi(t)=1$ if $t \ge 1$. The above formal discussion applies to this setting after replacing $\mathcal E_\gamma$ by its positive-phase version
$$ \mathcal E_{0, \gamma} (v):= \int_\Omega |\nabla v|^2 + (v^+)^{-\gamma} \, dx.$$
When $ \gamma=0$ this energy is the Alt-Caffarelli energy \cite{AC} that appears in the study of jet flows and cavitation problems, which was studied extensively in the literature (see \cite{CS}), while the exponent $\gamma=-1$ corresponds to another classical free boundary problem, the (no sign) obstacle problem. For the case $\gamma=0$, a theorem is the spirit of Theorem 1.1 was established recently by Audrito and Serra in \cite{AS}. Precisely they showed that nonnegative global minimizers of $J_{0,\gamma}$ with potential $W_{0,\gamma}$ inherit the rigidity properties of nonnegative global minimizers of the limiting Alt-Caffarelli energy.

The paper is organized as follows. In Section 2 we discuss some preliminary results such as one-dimensional solutions, the construction of radial barriers and we recall the results of \cite{ACKS} concerning minimizers of $\mathcal F$. In Section 3 we show that there are no global minimizers which have only one of the sets $\{u>1\}$ and $\{u<-1\}$ nonempty. In Section 4 we prove that global minimizers must be asymptotically linear at infinity and establish the main theorem from an improvement of flatness result. Finally in Section 5 we discuss the case of global monotone solutions with graphical 0 level set.

\section{Preliminaries}

We consider minimizers of the energy functional
\begin{equation}\label{J2}
J(u,\Omega):=\int_\Omega \frac{|\nabla u|^2}{2} + W(u) \, dx.
\end{equation}
for a potential function $ W: \mathbb R \to [0, \infty)$ that satisfies the following hypotheses:

a) $W=0$ outside the interval $[-1,1]$,

b) in the interval $[-1,1]$, $W$ is a $C^2$ function and
$$ W(\pm 1)=0, \quad W' (\pm 1)=0, \quad W''(\pm 1) >0,$$
$$ W'>0 \quad \mbox{in $(-1,0)$,} \quad W'<0 \quad \mbox{in $(0,1)$,} \quad W''(0)<0.$$ 
By elliptic regularity, minimizers are of class $C^{2,\alpha}$ for any $\alpha <1$, and they satisfy the Euler-Lagrange equation 
\begin{equation}\label{ELe}
\triangle u=W'(u).
\end{equation}

We state our main result as follows.

\begin{thm}\label{GM}
Let $u$ be a global minimizer of \eqref{J2} in $\mathbb R^n$. Then $u$ is one-dimensional if one of the two conditions hold

a) $u$ is bounded below or $u$ is bounded above and $n \le 7$,

b) $u$ is unbounded below and above, and $u=o(|x|^2)$.
\end{thm}

As remarked in the Introduction the subcase when $|u| \le 1$ in part a) of the theorem was already proved in \cite{S1,S3}, see Theorem \ref{DG}. In this paper we focus on the remaining two cases in which only one, or both of the sets
\begin{equation}\label{2sets}
\mbox{$\{u>1\}$ and $\{u<-1\}$ }
\end{equation}
are nonempty. We give a formal definition distinguishing between these two cases. 

\begin{defn} \label{bcases} We say that 

\noindent
i) $u$ is { \it unbounded on one side} if exactly one of the two sets in \eqref{2sets} is nonempty.

\noindent
ii) $u$ is { \it unbounded on both sides}  if the two sets in \eqref{2sets} are nonempty.
\end{defn}

{\bf One-dimensional solutions.}

First we introduce the family of non-constant one dimensional solutions that satisfy the ODE
\begin{equation}\label{1ODE}
u''=W'(u).
\end{equation}
After multiplying with $u'$ and integrating we find
$$ (u')^2=2W(u) + \lambda,$$
for some constant $\lambda$. Then, in an interval where $u$ is increasing, 
$$u'= \sqrt{2W(u) + \lambda},$$
and $u$ is obtained (up to a translation) as the inverse of
\begin{equation}\label{Gl}
G_\lambda(s):=\int_0^s \frac{1}{\sqrt{2W(s) + \lambda}} \, ds, \quad \quad u=G_\lambda^{-1}.
\end{equation}
 
 If $\lambda>0$, then $G_\lambda$ is well defined on $\mathbb R$, and  in each interval $(-\infty, -1]$ and $[1, \infty)$ it is linear of slope $\lambda^{-1/2}$. The corresponding solution $u$ is strictly increasing, and it is linear of slope $\lambda^{1/2}$ in the two intervals where $u>1$ and $u<-1$.  
 
 If $\lambda=0$ then $G_0$ is defined in $(-1,1)$, and the quadratic behavior of $W$ near $\pm 1$ implies that $G_0$ approaches the asymptotes $s = \pm 1$ at a logarithmic rate. Then the corresponding $u$ is strictly increasing and has limits $\pm 1$ at $\pm \infty$.
 
 If $\lambda <0$, then $G_\lambda$ is defined in a compact interval of $(-1,1)$ and the corresponding $u$ is periodic. These solutions are no longer global minimizers of the energy $J$. 
 
 \begin{defn}\label{Ua}
    For any $a>0$, we define $U_{a}$ as the solution to the ODE \eqref{1ODE} such that $U_a(0)=0$, and $U_a$ has slope $a$ outside the horizontal strip $|U_a|<1$. Equivalently, $$U_a:= G_{\lambda}^{-1}, \quad \quad \lambda=a^2,$$
    with $G_\lambda$ as in \eqref{Gl}.
    
    We also define $U_0=G_0^{-1}$ as the increasing bounded solution to \eqref{1ODE}.
\end{defn}

If $a \ge \delta$ then, by \eqref{Gl}, both $U_a^{-1}$ and the derivative of $U^{-1}_{\gamma a}$ with respect to the parameter $\gamma \in (\frac 12,2)$ are uniformly bounded in the interval $[-1,1]$. This means that in this interval
$$U_{\gamma a}^{-1}(s)= \frac 1 \gamma U^{-1}_{a}(s) + O(\gamma-1),$$
with $O(\cdot)$ depending only on $\delta$. The equality can be extended to all $s \in \R$ since the $U_a^{-1}$ is linear of slope $a^{-1}$ outside $[-1,1]$.
By taking $s=U_{\gamma a}(t)$ we obtain the following relation between one-dimensional solutions of different slopes 
\begin{equation}\label{usigma}
U_a^{-1} \circ U_{\gamma a} (t)= \gamma \, t + O(\gamma -1) \quad \quad \forall \, t \in \R.
\end{equation}

\

{\bf Growth of unbounded solutions.}  

Let us assume that $u$ is unbounded on both sides, that is 
\begin{equation}\label{two-ph}
\mbox{$\{u>1\}$ and $\{u<-1\}$ are both nonempty.}
\end{equation}
Then $(u-1)^+$ and $(u+1)^-$ are both harmonic in their positivity sets and have disjoint nonempty supports. An application of the Alt-Caffarelli-Friedman monotonicity formula implies  that $u$ must grow at least linearly at infinity. Below we give some of the details of this fact.

We first recall the ACF monotonicity formula \cite{ACF}.

\begin{thm}[ACF monotonicity formula] \label{ACF}
    Let $v_{+},v_{-}$ be continuous function defined in $\mathbb{R}^{n}$ such that
    \begin{equation*}
        v_{\pm}\geq0,\quad\triangle v_{\pm}\geq0,\quad v_{+}\cdot v_{-}=0.
    \end{equation*}
    For any $r>0$, if we define
    \begin{equation*}
     \Phi(r):=\frac{1}{r^{4}}I_{+}(r)\cdot I_{-}(r), \quad \quad  \quad I_{\pm}(r):=\int_{B_{r}}\frac{|\nabla v_{\pm}|^{2}}{|x|^{n-2}}dx,\quad 
    \end{equation*}
    Then $\Phi(r)$ is non-decreasing in $r$.
\end{thm}

Notice that 
$$ I_+(r) \le C  \|v_+\|^2_{L^\infty(B_{2r})},$$
for some constant $C$ that depends on $n$. Indeed, if $\eta$ denotes a cutoff function which is 1 in $B_r$ and vanishes outside $B_{2r}$, then
$$|\nabla v_+|^2 \le \triangle  \, v_+^2, \quad \quad \triangle (\eta |x|^{2-n}) \le C r ^{-n},$$
and we have
$$I_+(r) \le \int_{B_{2r}} \eta |x|^{2-n} \triangle (v^2_+) dx = \int_{B_{2r}} v_+^2 \triangle (\eta |x|^{2-n}) dx \le C r^{-n} \int_{B_{2r}} v_+^2 dx.$$
As a consequence we have the following growth lemma for solutions which are unbounded on both sides.

\begin{lem}\label{l25}
Assume that $u$ is a global solution to \eqref{ELe} which satisfies \eqref{two-ph}. Then
$$ \|u\|_{L^\infty (B_r)} \ge \delta_0 \, r \quad \quad \quad \mbox{for all large $r$,}$$ 
for some small constant $\delta_0>0$.
\end{lem}

\begin{proof} We apply the ACF monotonicity formula for $(u-1)^+$ and $(u+1)^-$. 

By \eqref{two-ph} there exists a large $r_0$ such that $\Phi(r_0) >0$. Then $$\Phi(r) \ge \delta:= \Phi(r_0),$$ for all large $r$, thus

\begin{align*} \delta r^4 & \le I_+(r) \cdot I_-(r)\\
& \le C \|(u-1)^+\|^2_{L^\infty(B_{2r})} \cdot \|(u+1)^-\|^2_{L^\infty(B_{2r})} \\
& \le C \|u\|_{L^\infty(B_{2r})}^4.
\end{align*}

\end{proof}

Concerning solutions that are unbounded on one side we will show in the next section that they must grow at most like $O(|x|^{1/2})$ at infinity. Then we will also establish a Gamma-converge result to a Dirichlet/perimeter functional whose properties are discussed below. 

\

{\bf Minimizers of the ACKS functional.}

Athanasopoulous, Caffarelli, Kenig, Salsa in \cite{ACKS}  introduced the Dirichlet-perimeter functional $\mathcal F$ defined below. Here we consider only the case of nonnegative minimizers $u$ which is relevant to our setting. The functional $\mathcal F$ acts on the space of admissible pairs $(u,E)$ consisting of functions $u \ge 0$ and measurable sets $E \subset \Omega$ which have the property that $u=0$ a.e. on $E$,
$$ \mathcal A_0(\Omega):=\{(u,E)| \quad u \in H^1(\Omega),\quad \mbox{$E$ Caccioppoli set, $u \ge 0$ in $\Omega$, $u=0$ a.e. in $E$} \}.$$
The functional $\mathcal F$ is given by the Dirichlet-perimeter energy
$$\mathcal F_\Omega(u,E)= \int_\Omega |\nabla u|^2 dx + Per_\Omega(E),$$
 where $Per_{\Omega}(E)$ represents the perimeter of $E$ in $\Omega$
 \begin{align*}
 Per_\Omega(E)&=\int_\Omega |\nabla \chi_E| \\
 &=\sup \,  \int_\Omega \chi_E \, div \, g \, dx \quad \mbox{with} \quad g \in C_0^\infty(\Omega), \quad |g| \le 1.  
 \end{align*}
If $(u,E)$ is a minimizing pair for $\mathcal F$, then near a point where $\p E$ is a regular surface, $u$ is harmonic in the complement $E^c$ and vanishes continuously on $\p E$, while along $\p E$ the {\it free boundary} condition 
$$ (u_\nu^+)^2 = H_{\p E}$$
holds. Here $\nu$ denotes the outer normal to $\p E$, and $H_{\p E}$ the mean-curvature of $\p E$. 

The natural scaling that leaves the functional $\mathcal F$ invariant is
$$(u,E) \mapsto (u_\lambda,E_\lambda) \quad \mbox{with} \quad \quad \quad  u_\lambda (x) := \frac{u(\lambda x)}{\lambda^{1/2}}, \quad \quad E_\lambda:= \frac 1 \lambda \,  E.$$

We recall here the interior regularity of minimizers of $\mathcal F$ established by Athanasopoulous, Caffarelli, Kenig, Salsa (see Theorems 4.1 and 5.1 in \cite{ACKS}). 

\begin{thm}[\cite{ACKS}]Let $(u,E)$ be a minimizing pair for $\mathcal F$ in $B_1$. Then $u$ is Lipschitz continuous and $\p E$ is a regular surface except on a closed singular set of Hausdorff dimension $n-8$. Moreover, if $0 \in \p E$ then 
$$ \|u\|_{C^{0,1}(B_{1/2})} \le C,$$
with $C$ a constant depending only on $n$.
\end{thm}

In \cite{DS3}, De Silva and the first author established a Weiss-type monotonicity formula for $\mathcal F$, and then characterized global minimizers of $\mathcal F$ in low dimensions (see Proposition 5.3 in \cite{DS3}):  

\begin{thm}[\cite{DS3}]\label{P1'}
Assume $n \le 7$ and $(u,E)$ is a global minimizer for $\mathcal F$ with $0 \in \p E$. Then $u \equiv 0$ and $E$ is a half-space.
\end{thm}

We also state the continuity up to the boundary of the minimizers of $\mathcal F$. Although this was not discussed in \cite{ACKS}, the argument is standard and we briefly sketch it for completeness.  

\begin{lem}[Boundary continuity]\label{ACKS continuity}
    Let $(u,E)$ be a minimizer of the functional $\mathcal F$ in $B_1$ with continuous boundary data for $u$ on $\p B_1$. Then $u$ is continuous in $\overline{B_{1}}$.
\end{lem}

\begin{proof} We need to trap $u$ between upper and lower barrier functions that guarantee the continuity up to the boundary. 

The upper barrier is simply the harmonic replacement $h$ of $u$ in $B_1$ with the same boundary data. Indeed, notice that the pair $(\min \{u,h\},E)$ is admissible and lowers the energy of $(u,E)$ if $\{u >h\}$ is nonempty.

The lower barriers are given by suitable truncations and translations of the fundamental solution. Precisely, the function
$$ V(x):= \sigma \cdot (|x-x_0|^{2-n}- r^{2-n})^+, \quad \quad \sigma \ge C r^{n- \frac 32},$$
is a comparison subsolution for minimizers of $\mathcal F$ in the annulus $B_{r}(x_0) \setminus B_{r/2}(x_0)$, in the sense that $\{u<V\}$ cannot be compactly included in this annulus. Otherwise we replace $u$ by $V$ in the set $\{u<V\}$ and $E$ by $E \setminus B_r(x_0)$, and it is not difficult to check that the lower bound on $\sigma$ guarantees the new pair decreases the energy, (see \cite{ACKS}).

\end{proof}

\

{\bf Radial barriers.}

We conclude this section by constructing useful radial barriers which will be used in the proofs. 

\begin{lem}\label{rad} Let $R \ge 2$. There exits a radial function $V_R(x)$ defined in the annulus $B_{2R} \setminus B_{R/2}$ such that
 $$ \triangle V_R > W'(V_R),$$
 and
$$V_R = -1 \quad \mbox {on $\p B_{2R}$,} \quad \quad 1 < V_R \le C R^{1/2}  \quad \mbox {in $\overline B_{R} \setminus B_{R/2}$},$$
with $C$ a constant depending only on $n$ and $W$.

\end{lem}

\begin{proof} We construct $V_R$ by rotating the graph of an increasing function $g(t)$, defined in the interval $[0, \frac 32 R]$, with respect to the axis $t=2R$. It suffices to require that 
\begin{equation}\label{ose1}
g'' - \frac{2(n-1)}{R} g' > W'(g),
\end{equation}
and $$g(0) = -1, \quad 1\le g \le C \sqrt R \quad \mbox{on $ [\frac 12 R,\frac 32 R]$}.$$
Whenever we deal with a second order autonomous ODE involving an increasing function $g$ as in \eqref{ose1}, it is convenient to change variables by considering $g$ as an independent variable. More precisely, we express $g'$ in terms of the variable $s=g$ by the formula 
\begin{equation}\label{hdef}
 g'= \sqrt{2 h(g)}= \sqrt{2 h(s)},
 \end{equation}
which defines the function $h$ on the range of $g$, and by chain rule we find
$$g''= h'.$$
The function $g$ is recovered from $h$ (up to a translation) by the formula
$$g^{-1}= \int \frac {1} {{\sqrt {2h(\xi)}}} \, \,  d\xi.$$
 With this change of coordinates \eqref{ose1} can be rewritten as
 \begin{equation}\label{ose2}
 h' - \frac{C_0}{R} \sqrt h > W',
 \end{equation}
 with $C_0$ a constant depending only on $n$. 
 
 We choose $h$ in $[-1, \infty)$ to be the Lipschitz function
$$h(s):=
\left\{
\begin{array}{l}
W(s) +  \frac{C}{R}(s+1) , \quad \quad \quad \quad \mbox{if} \quad |s| \le 1,\\

\

\\

\left( \sqrt{\frac{2C}{R}} + \frac{C_0}{R} (s-1)\right)^2, \quad \quad \quad \quad \quad  \mbox{if} \quad s > 1.\\
\end{array}
\right .
$$
Then we define $g$ through its inverse $g^{-1}$ by the equality $$g^{-1}(s)= \int_{-1}^s \frac {1} {{\sqrt {2h(\xi)}}} \, \,  d\xi,$$ and notice that $g(0)=-1$ by construction.

 The inequality \eqref{ose2} is clearly satisfied in the interval $[-1,1]$ provided that $C$ is chosen large depending on $C_0$ and $\max \, W$. It also holds on the other interval $[1,\infty)$, since $h$ was chosen to satisfy
 $$ h'= \frac {2C_0}{ R}  \, \sqrt h, \quad \quad h(1)= \frac{2C}{R},$$ 
 
 Using that near $\pm 1$
 $$W(s) \sim (1-s^2)^2,$$
 we easily find that
 $$ g^{-1}(1)=\int_{-1} ^{1} \frac {1} {{\sqrt {2h(\xi)}}} \, \,  d\xi \le C' \log R,$$ 
  thus $g(t) > 1$ if $t > C' \log R$.
  
  On the other hand
  $$ \int_1^ {1+C_1 \sqrt R} \frac {1} {{\sqrt {2h(\xi)}}} \, \,  d\xi \ge \int_0^{C_1} \frac{c R}{1+\sigma} d\sigma \ge \frac{3R}{2} ,$$
  if $C_1$ is sufficiently large. Thus $g (\frac 32 R) \le 1+ C_1 \sqrt R$.

\end{proof}

Next we construct a family of subsolutions in balls $B_R$ which are perturbations of the one-dimensional solution $U_a$ of slope $a$ (see Definition \ref{Ua}) and that are radial with respect to an axis at distance $R/\eps$ from the origin.

\begin{lem}\label{l210} Assume $a \ge \delta$ for some fixed $\delta$, and let $R \ge R_0(\delta)$. There exists a function $g$ defined in $[-2R,2R]$ which is a perturbation of the one-dimensional solution $U_a$,
$$g = U_a \circ \tau, \quad \quad \tau(t)=t + O(\frac \eps R) \cdot t^2,$$
such that the function $$ \Phi(x)=g(\frac R \eps -|x|),$$
is a subsolution on its domain of definition 
\begin{equation}\label{trp}
\triangle \Phi> W'(\Phi).
\end{equation}
The constant in $O(\cdot)$, and $R_0(\delta)$ depend only on $\delta$, $n$ and $\|W\|_{L^\infty}$.
\end{lem}

\begin{proof} First we remark that it suffices to prove the lemma for $a=1$. We reduce to this case after dividing $g$ and $U_a$ by the constant $a$, and then they satisfy the equations with rescaled potential
$$W_a(t)=a^{-2} W(a t ).$$ 
Notice that $W_a$ is bounded in terms of $\delta$ and $\|W\|_{L^\infty}$, and it vanishes outside a compact interval of length depending on $\delta$.

Assume now that $a=1$. We construct $g$ in $[-2R,2R]$ with $g(0)=0$, such that 
$$g''- C(n) \frac{\eps}{R} g' > W'(g).$$ 
As above we view $g=s$ as a variable and define $h$ by \eqref{hdef}. In terms of $h$ the inequality reads
\begin{equation}\label{abo}
 h'- C(n) \frac{\eps}{R} \sqrt{2h} \ge W'.
 \end{equation}
We define $h$ in $[-4R,4R]$ as
$$ h(s)=h_1(s) + C \frac  \eps R s, \quad \quad h_1(s):= W(s) + \frac 1 2,$$
with $C$ sufficiently large depending on $n$ and $\|W\|_{L^\infty}$ to guarantee that the inequality \eqref{abo} is satisfied. 
Since $h$ and $h_1$ are bounded above and below,
$$(2 h)^{-1/2} = (2h_1)^{-1/2} + O(\frac  \eps R s),$$
hence
$$ g^{-1}(s)=\int_0^s \frac{1}{\sqrt {2h(\xi)}} d\xi=U_1^{-1}(s) + O(\frac  \eps R s^2).$$
Since $g^{-1} (s) \sim s$, we may replace the $s^2$ in the error term by $[g^{-1}(s)]^2$, and then we obtain the result by plugging $s=g(t)$.

\end{proof}

\section{Unbounded solutions on one side}

In this section we prove Theorem \ref{GM} part a) and focus on the remaining case in which we assume the solution is bounded on one side and unbounded on the other. We show that no such solutions exist in dimension $n \le 7$.

We assume throughout that $u$ is a global minimizer of $J$ which satisfies part i) of Definition \ref{bcases}, that is, 
\begin{equation}\label{bcase1}
 \{u <-1\} = \emptyset, \quad \quad \{ u>1 \} \ne \emptyset. 
 \end{equation}
 The strong maximum principle implies that in fact $u >-1.$
 Notice that the function $u$ is harmonic in the set $\{u > 1\}$, so we may also assume that 
\begin{equation}\label{bcase2} 
 \{u \ge 1\} \ne \R^n, 
 \end{equation}
 since otherwise $u$ must be constant. Denote the boundary of $\{u>1\}$ by $\Gamma$, i.e.
 $$\Gamma:=\partial\{u>1\},$$
 and notice that $\Gamma \ne \emptyset$.

\subsection{Optimal growth and energy bounds} First we show that $(u-1)^+$ must grow at most as square root of the distance to $\Gamma$. 
 
 \begin{lem}[$C^{1/2}$ estimate]\label{one-phase holder}
    Assume $u$ is a global minimizer of $J$ and \eqref{bcase1} holds. 
    There exists a constant $C$ depending only on $n$ and $W$ such that
   $$(u-1)^+ (x) \le C \left ( dist(x,\Gamma) \right)^{1/2}.$$
   In particular $$[(u-1)^+]_{C^{1/2}(\mathbb{R}^{n})}\leq C,$$
   and if $u(0) \le 1$ then 
  $$ \|u \|_{L^\infty(B_R)} \le C R^{1/2}, \quad \quad \forall R \ge 1.$$
\end{lem}

\begin{proof} Assume that $B_R \subset \{u>1\}$ is tangent to $\Gamma$ at some point $x_0$. It suffices to show that
\begin{equation}\label{1001}
(u-1) (0) \le C R^{1/2},
\end{equation} for some large $C$.

If $R \le C_0$ then this follows easily from the Euler-Lagrange equation
$$\triangle u = W'(u) \quad \Longrightarrow \quad |\triangle u| \le C \quad \mbox{in $B_{2C_0}$}.$$
 Indeed, since $u \ge -1$ and $u(x_0)=1$, Harnack inequality gives $$\|u\|_{L^\infty(B_{\frac 32C_0})} \le C.$$ Then by interior gradient estimates
 $$ \|u\|_{C^{0,1}(B_{C_0})} \le C,$$
 which together with $(u-1)(x_0)=0$ gives an upper bound $CR$ which implies the desired upper bound in \eqref{1001}.

If $R \ge C_0$, assume by contradiction that \eqref{1001} does not hold. Then by Harnack inequality applied to $u-1 \ge 0$ in $B_R$ we find 
\begin{equation}\label{1002}
 u-1 \ge C' R^{1/2}  \quad \mbox {in} \quad B_{R/2},
 \end{equation}
with $C'$ sufficiently large. We claim this inequality implies $$u \ge V_R \quad \mbox{ in} \quad  B_{2R} \setminus B_{R/2},$$ 
with $V_R$ the subsolution constructed in Lemma \ref{rad}. We reached a contradiction since $u(x_0)=1$ and $V(x_0)>1$.

The claim is a consequence of the maximum principle by comparing $u$ with the continuous family of subsolutions $V_t$, with $t \in [2,R]$, in the domains of definition of $V_t$'s. Indeed, notice that \eqref{1002} and $u > -1$ imply $u > V_t$ for $t=2$, and also $u>V_t$ on $\p (B_{2t} \setminus B_{t/2})$ for all the other values of $t$. This means the strict inequality can be extended to the interior of the domains as well.

\end{proof}

Next we use the growth of $u$ and bound the energy of $u$ in large balls $B_R$. 
\begin{lem}[Energy estimate]\label{one phase energy bound}
    Assume that $u(0) \le 1$. Then
 $$J(u,B_{R})\leq C R^{n-1} \quad \mbox{for all $R\geq1$,}  $$
 with $C$ a constant that depends on $n$ and $W$.
\end{lem}

\begin{proof}
  Define $v(x)$ as
    $$ v(x):=-1 +  C_0 \min \left\{ (|x|-R)^+,\left [ (|x|-R)^+ \right ] ^\frac 12 \right\},$$
  and let $E$ denote the set
    $$ E:=\{u \ge v \}.$$
   By Lemma \ref{one-phase holder} we know that
   $$ B_{R}\subseteq E\subseteq B_{2R},$$
   provided that $C_0$ is chosen sufficiently large. We take $\min\{u,v\}$ as a competitor for $u$ in $B_{2R}$ and obtain
   $$ J(u,E) \le J(v,E) \quad \Longrightarrow \quad J(u,B_R) \le J(v,B_{2R}).$$
  We obtain the desired bound since $$J(v,B_{2R}) \le C R^{n-1}.$$
  
\end{proof}

A consequence of the energy bound is that the measure of the set $\{|u|<s\}$, for some fixed $s \in (0,1)$, grows at most like $R^{n-1}$ in $B_R$:
\begin{equation}\label{ules}
|\{|u|\leq s\}\cap B_{R}| \leq C(s) R^{n-1} \quad \quad  \forall R  \geq1,
\end{equation}
with $C(s)$ a constant depending also on $s$.

 \subsection{Density estimates}
 
 Next we derive density estimates for the sets $\{u<0\}$ and $\{u>0\}$.
 
 \begin{lem}[Density estimates]\label{density estimate}
  Assume that $u(0)=0$. Then for all $R \ge C$,  
    \begin{equation*}
        |B_{R}\cap\{u\geq0\}|\geq\delta R^{n},\quad|B_{R}\cap\{u\leq0\}|\geq\delta R^{n}.
    \end{equation*}
    for some small constant $\delta>0$ that depends only on $n$ and $W$.
\end{lem}
 
 \begin{proof} The proof of the density estimates in the bounded case $|u|<1$ in $\R^n$ is due to Caffarelli-Cordoba in \cite{CC} (see also \cite{S2,DFV}). The strategy is to compare the energy of $u$ 
 with that of the energy of an explicit function $v$ and then derive a discrete differential inequality involving the ``volume" $V(R)$ and ``area" $A(R)$ type quantities associated with $u$ defined as
 $$ V(R):=|\{ u \ge 0\} \cap B_R|, \quad \mbox{and} \quad A(R)= \int_{B_{R}}W(u)dx.$$

 We remark that since $u>-1$, the second density estimate for the sub-level set $\{u \le 0\}$ follows exactly as in \cite{CC}. It remains to prove the first density estimate, and the main difference with respect to the arguments in \cite{CC} is that now $u$ is not bounded above by 1. 
 
 Denote by
 $$ \omega(R):= V(R) + A(R),$$
and it suffices to show that 
 $$ \omega(R) \ge \delta R^n \quad \mbox{for a sequence $R=R_k \to \infty$, with $R_0 \le C$,} \quad R_{k+1}/R_k \le C,$$
 with $\delta$ small, and $C$ large, appropriate universal constants. Then the density estimate follows since $A(R) \le C R^{n-1}$ by Lemma \ref{one phase energy bound}.
  
We prove the claim by constructing the sequence $R_k$ inductively. 

By Lemma \ref{one-phase holder}, 
$$\|u\|_{L^\infty(B_2)} \le C, \quad \mbox{and} \quad |\triangle u| \le C \quad \quad \Longrightarrow \quad \|\nabla u\|_{L^\infty(B_1)}\leq C.$$ Since $u(0)=0$ this means that $A(1) \ge c_1$ for some small constant $c_1$, hence $\omega(1)\geq c_{1}$. Let $R_0=5T$ where $T>1$ is a fixed large constant to be specified later, and then we have     \begin{equation}\label{density assumption 1}
 \omega(R_{0})\geq\delta R_{0}^{n} \quad \mbox{provided that} \quad  \delta\leq c_{1}(5T)^{-n}.
    \end{equation}
 
 Now suppose that $A(R_{k})\geq\delta R_{k}^{n}$ for some $k\geq0$, and it suffices to show that there exists $R_{k+1}\leq 3 R_{k}$ so that 
 $$\omega(R_{k+1})\geq\delta R_{k+1}^{n}.$$ We distinguish two cases. 
 
 \smallskip
 {\it Case 1:}  In the annulus $$\mathbb{A}_{k}:=B_{3R_{k}}\setminus B_{R_{k}},$$ we have $$|\mathbb{A}_{k}\cap\{u\geq1\}|\geq\delta|\mathbb{A}_{k}|.$$ 
Then we choose $R_{k+1}=3R_{k}$. Clearly the contribution to the $V(R)$ term in the annulus $\mathbb{A}_{k}$ gives $\omega(R_{k+1})\geq\delta R_{k+1}^{n}$, which is the desired conclusion. 

\smallskip
{\it Case 2:} Assume that
\begin{equation}\label{case2}
|\mathbb{A}_{k}\cap\{u\geq1\}|<\delta|\mathbb{A}_{k}|.
\end{equation}
We choose $R_{k+1}:=R_{k}+T\le  \frac 65 R_{k}$ and denote for simplicity of notation 
$$r:=R_{k+1}.$$
It remains to show $\omega(r)\geq\delta r^{n}$.
   Since $(u-1)^{+}$ is subharmonic, we apply the mean value inequality in balls of radius $r/2$ centered at points on $\p B_{2r}$ and use \eqref{case2} to conclude that
    \begin{equation*}
        \|(u-1)^+\|_{L^{\infty}(\partial B_{2r})} \leq C \delta \|(u-1)^+\|_{L^{\infty}(\mathbb{A}_{k})} \leq C_{2} \delta \sqrt{r},
    \end{equation*}
    where in the last inequality we have used Lemma \ref{one-phase holder}.
    
   We define $v(x)$ in $B_{2r}$ as
    \begin{equation*}
        v(x)=\left\{\begin{aligned}
            -1+2e^{|x|-r},& \quad \quad \mbox{ if }|x|\leq r,\\
            1+C_{2}\delta \frac{|x|-r}{\sqrt{r}},& \quad \quad \mbox{ if }r\leq|x|\leq2r.
        \end{aligned}\right.
    \end{equation*}
    Notice that $v$ is Lipschitz and
   
    \noindent
    a)  $v \le 1$ in $B_r$, and $v \ge 1$ in $B_{2r} \setminus B_r$, 
   
     \noindent
    b) $v\geq u$ on $\partial B_{2r}$,
   
    \noindent
    c)
      \begin{equation}\label{J(v) outside}
         J(v,B_{2r}\setminus B_{r})\leq C\delta^2 r^{n-1},\quad J(v,B_{r-T})\leq Ce^{-2T}r^{n-1}.
    \end{equation}
    Let $E$ denote the set
    $$E:= \{u>v\} \cap B_{2r},$$
    and by minimality of $u$ we have 
    \begin{equation}\label{in1}
    J(u, \overline E) \le J(v, \overline E).
    \end{equation}
    Next we recall the classical inequality of Modica-Mortola for the energy $J(w, \overline E)$ of an arbitrary Lipschitz function $w$, and the perimeter of the intermediate level sets of $w$, which follows from the Cauchy-Schwartz inequality and the co-area formula:
     \begin{equation}\label{in2}
    J(w,\overline E) \ge \int_{\overline E} \sqrt{2W(u)} \, |\nabla u| dx \ge \int_{-1}^1 \mathcal H^{n-1}(\overline E \cap \{w=s\}) \sqrt{2 W(s)} ds.
    \end{equation} 
  We use this inequality for the functions $u$ and $v$. For $s \in [-1,1]$ we define
    $$E_{s}:=\{u>s>v\} \cap B_{2r}, $$ 
    and then
    \begin{equation*}
        \partial E_{s} \subset \partial_{u}E_{s}\cup\partial_{v}E_{s},
    \end{equation*}
    where $$\partial_{u}E_{s}:= \overline E\cap\{u=s\} \quad \mbox{ and} \quad  \partial_{v}E_{s}:=\overline E\cap\{v=s\}.$$ We multiply by $\sqrt{2 W(s)}$ the isoperimetric inequality for the sets $E_s$
    \begin{equation*}
        |E_{s}|^{\frac{n-1}{n}}\leq C \, [\mathcal{H}^{n-1}(\partial_{u}E_{s})+\mathcal{H}^{n-1}(\partial_{v}E_{s})],
    \end{equation*}
   and then integrate in $s$, which together with \eqref{in1}-\eqref{in2} gives
  $$ \int_{-1}^{1}|E_{s}|^{\frac{n-1}{n}}\sqrt{2W(s)}ds\leq C[J(u,\overline E)+J(v, \overline E)]\leq 2C J(v,\overline E).$$
   Notice that when $s$ belongs to the interval $(-1+2 e^{-T}, 0)$ then $$B_{r-T} \subset \{v<s\}, \quad \mbox{and} \quad \{u \ge 0\}  \subset \{s < u\},$$ hence
   $$B_{r-T} \cap \{u \ge 0\} \subset E_s,$$
   which means $$ V(r-T) \le |E_s|.$$ 
  Then the integral inequality above implies
  $$ (V(r-T) )^\frac{n-1}{n} \le C J(v, \overline E).$$
  On the other hand using that $W$ is increasing in the interval $[-1, -1+2 e^{-T}]$ we have
  \begin{align*}
  A(r-T) =\int_{B_{r-T}} W(u) dx  & \le  J(u, \overline E) + \int_{B_{r-T}\setminus E} W(v) dx \\
  & \le J(v, \overline E) + J(v, B_{r-T}) \\
 &  \le J(v, \overline E) + C e^{-2T}r^{n-1}.
  \end{align*}
   We combine the last two inequalities and use that $A(r-T) \ge A(1) \ge c_1$ is bounded below to conclude that
   \begin{equation}\label{om1}
   (\omega(r-T)) ^\frac{n-1}{n} \le C J(v, \overline E) + C e^{-2T}r^{n-1}.
   \end{equation}
   Now we estimate $J(v, \overline E)$, and by \eqref{J(v) outside} we have
    \begin{equation}\label{om2}
        J(v,\overline E)\leq J(v, \overline E\cap(B_{r}\setminus B_{r-T}))+C r^{n-1}(\delta^2 +e^{-2T}),
    \end{equation}
    In the set $\overline E\cap(B_{r}\setminus B_{r-T}) $ we use that 
    
    \noindent
    a) $u\ge v$,
   
       \noindent
b) $W$ is increasing near $-1$,
   
    \noindent
    c) $W(v)\sim (1+v)^2 \sim |\nabla v|^2$ if $v \le 0$,
    
    \noindent 
    and find
    $$|\nabla v|^{2} + W(v) \leq C(W(v)+\chi_{\{v\geq0\}}) \le C' (W(u)+\chi_{\{u\geq0\}}).$$ Therefore
    \begin{equation*}
        J(v, \overline E\cap(B_{r}\setminus B_{r-T}))\leq C[\omega(r)-\omega(r-T)],
    \end{equation*}
   which combined with \eqref{om1}-\eqref{om2} gives
    \begin{equation*}
       c (\omega(r-T))^{\frac{n-1}{n}} -C r^{n-1}(\delta^2+e^{-2T}) \le \omega(r)-\omega(r-T) ,
    \end{equation*}
    for universal constants $c$ small and $C$ large.
  Using that $\omega(r-T)$ satisfies the induction hypothesis we obtained the desired inequality $$ \delta r^n  \le \omega(r),$$ provided that     
   \begin{equation}\label{density assumption 2}
     \delta T \le   c \delta^{\frac{n-1}{n}}-C (\delta^2 + e^{-2T}).
    \end{equation}
   Finally we remark that indeed it is possible to choose constants $\delta$ and $T$ so that both requirements \eqref{density assumption 1} and \eqref{density assumption 2} are satisfied. For this we first take $\delta^\frac 1n T$ to be sufficiently small universal, and then choose $T$ sufficiently large.
   
   \end{proof}

\begin{rem} \label{rema}The density estimates remain valid under the assumption that $u(0)$ belongs to a compact interval of $(-1,1)$ provided that $R$ is sufficiently large. More precisely, if we assume that $$|u(0)|<s, \quad \mbox{for some $s<1$,}$$ then the proof above shows that $$|\{u \ge u(0) \} \cap B_R| \ge \delta R^n, \quad \mbox{and} \quad  |\{u \le u(0) \}\cap B_R| \ge \delta R^n,$$ if $R \ge C$, for some constants $\delta$ and $C$ that depend also on $s$. It is not difficult to see that, in view of \eqref{ules}, the constant $\delta$ can be actually chosen universal, independent of $s$.

\end{rem}

\subsection{$\Gamma$ convergence} 

Next we show that proper rescalings of a global minimizer $u$ of $J$ converge to a minimizer of the ACKS functional. We rescale $u$ in a large ball $B_R$ differently according to the regions where $|u|<1$ and $|u| \ge 1$ and create a pair $(v, \mathcal E)$ defined in $B_1$ as follows
\begin{equation}\label{rp}
 (v, \mathcal E):=\left(R^{- \frac 12} \cdot (u-1)^+(Rx), \min \{u(Rx),1\} \right).
 \end{equation}
Notice that the energy of $u$ in $B_R$ can be expressed in terms of the pair $(v, \mathcal E)$ as
$$ J (u, B_R)=R^{n-1} \int_{B_1} \frac{|\nabla v|^{2}}{2}+  \frac{|\nabla\mathcal{E}|^{2}}{2R}+ R \,  W(\mathcal{E}) \, \, dx. $$

\begin{defn}
    We say a pair of functions $(v,\mathcal{E})$ is admissible in an open set $\Omega$, denoted as $(v,\mathcal{E})\in\mathcal{A}(\Omega)$, if
    \begin{itemize}
        \item[(1)]  $v, \mathcal E \in H^{1}(\Omega)$,

        \item[(2)]   $v \geq0$ and $-1\leq\mathcal{E}\leq1$,
        
          \item[(3)] $\{v>0\} \subset \{ \mathcal E=1\}$.
    \end{itemize}
\end{defn}
\begin{defn}
    If $(v,\mathcal{E})\in\mathcal{A}(\Omega)$ and $\eps>0$, we define the functional
    \begin{equation*}
        J_{\eps}(v,\mathcal{E},\Omega):=\int_{\Omega}\frac{|\nabla v|^{2}}{2} + \eps \, \frac{|\nabla\mathcal{E}|^{2}}{2}+ \frac 1 \eps W(\mathcal{E})\, dx.
    \end{equation*}
 \end{defn}
Clearly, if $u$ is a minimizer in $B_R$ then the corresponding pair $(v, \mathcal E)$ minimizes $J_\eps$ in $B_1$ with $\eps=R^{-1}$. Conversely, given a minimizing pair $(v,\mathcal E)$ in $B_1$ we can find a corresponding minimizer $u$ of $J$ in $B_R$ with $R=\eps^{-1}$.

 We establish the Gamma-converge of the functionals $J_\eps$ as $\eps \to 0$, to the ACKS functional
  \begin{equation*}
        I(v,E,\Omega):=\int_{\Omega}\frac{|\nabla v|^{2}}{2}dx+c_0 \cdot Per_\Omega(E), 
    \end{equation*}
 with $c_0$ the constant 
 $$c_0:= \int_{-1}^1 \sqrt{2W(s)} \, ds.$$
 Notice that after multiplying $v$ by a constant this functional coincides with a multiple of the energy $\mathcal F$ discussed in Section 2. We recall that the functional $I$ acts on the space of admissible pairs $(v,E) \in \mathcal A_0(\Omega)$ consisting of functions $v \ge 0$ and measurable sets $E \subset \Omega$ which have the property that $v=0$ a.e. on $E$:
$$ \mathcal A_0(\Omega):=\{(v,E)| \quad v \in H^1(\Omega),\quad \mbox{$E$ Caccioppoli set, $v \ge 0$ in $\Omega$, $v=0$ a.e. in $E$} \}.$$

We state the precise results.   

\begin{thm}[$\Gamma$-convergence] \label{TGC}As $\eps \to 0$ the functionals $J_\eps$ $\Gamma$-converge to $I$ in the following sense:

a) (lower semicontinuity) if $\eps_k \to 0$ and 
\begin{equation}\label{con1}
v_k \to v \quad \mbox{ in} \quad L^2(\Omega), \quad \quad {\mathcal E}_k \to \chi_{E^{c}}-\chi_{E} \quad \mbox{ in} \quad L^1(\Omega),
\end{equation}
then
$$ \liminf J_{\eps_k}(v_k, \mathcal E _k, \Omega) \ge I(v,E,\Omega).$$

b) (approximation) given $(v,E) \in \mathcal A_0(\Omega)$ with $u$ continuous in $\overline \Omega$, there exists a sequence $(v_k, \mathcal E_k) \in \mathcal A(\Omega)$ such that \eqref{con1} holds and
$$J_{\eps_k}(v_k, \mathcal E _k, \Omega) \to I(v,E,\Omega).$$
\end{thm}

The next theorem establishes the compactness of minimizers.

\begin{thm}[compactness]\label{gamma convergence}
    Assume that $(v_k,\mathcal{E}_k)\in\mathcal{A}(B_{1})$ are minimizers of $J_{\eps_k}(\cdot,B_{1})$, such that $$J_{\eps_k}(v_k,\mathcal{E}_k,B_{1}) + \|v_k\|_{L^{2}(B_{1})} \le M $$ for some fixed $M>0$. Then there exists a pair $(v,E)\in\mathcal{A}_0(B_{1})$ such that up to subsequences
    $$ v_k \to v \quad \mbox{in} \quad C^\alpha_{loc}(B_1), \quad \quad \mathcal E_k \to \chi_{E^{c}}-\chi_{E} \quad \mbox{in} \quad L^1_{loc}(B_1),$$
    and $(v,E)$ minimizes the ACKS functional $I(\cdot,B_1)$.
\end{thm}

A direct consequence of Theorem \ref{gamma convergence} is the uniform convergence of the blowdowns of $u$ to a global minimizer of the functional $I$.

\begin{cor}\label{c39} Let $u$ be a global minimizer satisfying \eqref{bcase1} and $u(0) \in (-1,1)$. Then, along subsequences of $R_k \to \infty$, the rescaled pairs $(v, \mathcal E)$ defined in \eqref{rp} converge on compact sets as above to a limiting pair $(v,E)$ which is a global minimizer of $I$. Moreover, $0 \in \p E$, and the rescaled level curves $R_k^{-1} \cap \{u=u(0)\}$ converge uniformly on compact sets to $\p E$.  
\end{cor}

Indeed, by the density estimates Lemma \ref{density estimate} (see Remark \ref{rema}), we find that $0 \in \p E$. The uniform convergence of the rescaled level curves follows from the density estimates and the $L_{loc}^1$ convergence of $\mathcal E_k$ to $\chi_{E^c} - \chi_E$.

\begin{proof}[Proof of Theorem \ref{GM} part a)] 

We show that there are no minimizing solutions satisfying \eqref{bcase1}-\eqref{bcase2} in dimension $n \le 7$. 

After a translation we may assume that $u(0) \in (-1,1)$. By Theorem \ref{P1'}, any blowdown limit $(v,E)$ as in Corollary \ref{c39} has the form $v\equiv 0$ and $E$ is a half-space. Then the level set $\{u \le u(0)\}$ is asymptotically flat at infinity, in the sense that:

\smallskip

For any $\delta>0$, and all $R\ge C(\delta,u)$ sufficiently large
\begin{equation}\label{inc}
\{x \cdot \nu_R \le -\delta R \} \subset \{u\le u(0) \} \subset \{x \cdot \nu_R \le \delta R \} \quad \mbox{in $B_R$,}
\end{equation}
where $\nu_R$ denotes a unit direction that depends on $R$.

\smallskip

In turn, this property implies that as $R \to \infty$, the quantity
$$M(R):= \max_{B_R} (u-1)^+,$$
grows faster than any power $R^\alpha$ with $\alpha<1$. We reached a contradiction since Lemma \ref{one-phase holder} states that $M(R) \le C R^{1/2}$. 

In order to prove the claim, it suffices to show that
$$M(\delta R) \le \delta^\alpha M(R),$$
since by assumption \eqref{bcase1}, $M(R)$ is positive for large $R$. We use the inclusions \eqref{inc} and compare in $B_R \cap \{x \cdot \nu_R \ge -\delta R \}$ the subharmonic function $(u-1)^+$ with the harmonic function $w$ that vanishes on the flat part of the boundary where $\{x \cdot \nu_R =0\}$ and with $w=M(R)$ on the part of the boundary in $\p B_R$. We have $(u-1)^+ \le w$ and the harmonic function $w$ satisfies the bound 
$$w \le C \delta \cdot M(R) \quad \mbox{in $B_{\delta R}$,}$$
with $C$ a constant that depends only on the dimension $n$. We get the desired inequality by choosing $\delta$ sufficiently small depending on $\alpha$ and $n$.

\end{proof}

The remaining of the section is dedicated to the proofs of Theorem \ref{TGC} and Theorem \ref{gamma convergence}.

\
 
\begin{proof}[Proof of Theorem \ref{TGC}] a) We remark that the limiting pair $(u.E)\in \mathcal A_0(B_1)$ is admissible due to the a.e. pointwise convergence properties along subsequences. 

The proof of the lower semicontinuity is straightforward. We have
\begin{align}\label{Hin}
\nonumber \int_{\Omega}\eps_k \, \frac{|\nabla \mathcal{E}_k|^{2}}{2}+ \frac {1}{ \eps_k} W(\mathcal{E}_k)\, \, dx & \ge \int_\Omega \sqrt{2 W(\mathcal E_k)} |\nabla \mathcal E_k| dx \\
&=\int _\Omega |\nabla H(\mathcal E_k)| dx,
\end{align}
where $H$ denotes an antiderivative of $\sqrt{2W}$. Now the result follows from the lower semicontinuity property of the Dirichlet integrals of the $v_k$'s and of the BV norms of the  functions $H(\mathcal E_k)$'s. Notice that $H(\mathcal E_k)$ converges in $L^1_{loc}$ to $H(\chi_{E^{c}}-\chi_{E})$ and
$$ \int _\Omega |\nabla H(\chi_{E^{c}}-\chi_{E})| dx = c_0 Per_\Omega(E).$$
 
b) In view of the lower semicontinuity it suffices to construct an approximating sequence that satisfies
 $$\limsup J_{\eps_k}(v_k,\mathcal E_k,\Omega) \le  I (v,E, \Omega).$$
 Fix $\delta >0$ small. First we approximate $E$ in $\Omega$ by a smooth set $F \subset \R^n$ which is included in the open set $\{v< \delta\}$ in $\Omega$. Notice that since $v$ is assumed to be continuous up to the boundary in $\overline \Omega$, the set $\{v< \delta\}$ contains the intersection of a neighborhood of $E$ in $\R^n$ with the domain $\Omega$. By Lemma 1 in Modica \cite{M}, there exists a smooth set $F \subset \R^n$ which approximates the Caccipolli set $E$ in $\Omega$ in the following sense:
 $$ F \cap \Omega \subset \{ u < \delta \} ,$$
 $$ \|\chi_{F \cap \Omega} - \chi_E\|_{L^1} \le \delta, \quad \quad Per_{\Omega}(F) \le Per_{\Omega}(E) + \delta,$$
   $$\mathcal H^{n-1}(\partial F \cap \p \Omega) =0.$$
In view of this, it suffices to prove the result with $E$ replaced by $F$ and $v$ replaced by $\tilde v:=(v-2\delta)^+$ which approximates $v$ in $H^1(\Omega)$. Notice that in $\overline \Omega$, the function $\tilde v$ vanishes in a $\eta$-neighborhood of $F$ for some small $\eta>0$.

Now the construction is standard. Let $d$ denote the signed distance to $\p F$, with $d<0$ in $F$ and $d >0$ in $F^c$, and define $\mathcal E_k$ in $\Omega$ as
 $$ \mathcal E_k:= \varphi(d) g(d/\eps_k) + (1-\varphi(d)) \sgn (d),$$
 where
 
 \noindent a) $\varphi$ is a cutoff function supported in $[- \frac 12 \eta, \frac 12 \eta]$ with $\varphi=1$ in $[- \frac 14 \eta, \frac 14 \eta]$, 
 
 \noindent b) $g$ is the one dimensional solution,
 
  \noindent c) $\sgn$ denotes the signed function. 
  
  Then $(\mathcal E_k, \tilde v) \in \mathcal A(\Omega)$ is admissible, and it is easy to check that
  $$\lim_{k \to \infty} \int_{\Omega}\eps_k \, \frac{|\nabla \mathcal{E}_k|^{2}}{2}+ \frac {1}{ \eps_k} W(\mathcal{E}_k) \, \, dx = c_0 Per_{\Omega} (F).$$
\end{proof}

Before we proceed with the proof of Theorem \ref{gamma convergence} we need to establish a glueing procedure between two admissible pairs in $\mathcal A(B_1)$ that are sufficiently close in an annular region.

\begin{lem}\label{gluing}
    Let $U:=B_{1-\sigma}\setminus B_{1-4\sigma}$ for some $\sigma>0$ fixed, and let $\eps_k \to 0$. Assume that $$(v_k,\mathcal{E}_k)\in\mathcal{A}(B_{1}),  \quad \quad (u_k,\mathcal{F}_k)\in\mathcal{A}(B_{1-\sigma}),$$ and
    \begin{equation*}
        \lim_{k\to \infty} \left({\|v_k-u_k\|_{L^{2}(U)}+\|\mathcal{E}_k-\mathcal{F}_k}\|_{L^{1}(U)} \right) = 0.
    \end{equation*}
    Then there exists $(w_k,\mathcal{G}_k)\in\mathcal{A}(B_{1})$ such that
  $$ \mbox{ $(w_k,\mathcal{G}_k)=(u_k,\mathcal{F}_k)$ in $B_{1-4\sigma}$ and $(w_k,\mathcal{G}_k)=(v_k,\mathcal{E}_k)$ in $B_{1-\sigma}^{c}$,}$$
  and
        $$\limsup_{k\to\infty}J_{\eps_k}(w_k,\mathcal{G}_k,U) \leq 3 J_{\eps_k}(u_k,\mathcal{F}_k,U)+3J_{\eps_k}(v_k,\mathcal{E}_k ,U).$$
\end{lem}

\begin{proof}
    The key point is to construct $(w_k, \mathcal{G}_k)$ in $U$. We first construct the function $w_k$ and then construct $\mathcal{G}_k$.

\smallskip

     {\it Step 1:}  Let $\eta_{1},\eta_{2},\eta_{3}$ be a partition of unity in $B_{1}$ so that:
        
        \noindent
         1) $supp(\eta_{1})=B_{1-3\sigma}$, and $\eta_{1}=1$ in $B_{1-4\sigma}$,
         
         \noindent
        2) $supp(\eta_{2})=B_{1-\sigma}\setminus B_{1-4\sigma}$, and $\eta_{2}=1$ in $B_{1-2\sigma}\setminus B_{1-3\sigma}$,

        \noindent
        3) $supp(\eta_{3})=B_{1-2\sigma}^{c}$, and $\eta_{3}=1$ in $B_{1-\sigma}^{c}$,
       
        \noindent
        4) $\|\nabla\eta_{i}\|_{L^{\infty}}\leq C\sigma^{-1}$.
        
        We define $w_k$ in $B_{1}$ as
        \begin{equation}
            w_k=\eta_{1}u_k+\eta_{2}\min\{u_k,v_k\}+\eta_{3}v_k.
        \end{equation}
        Then clearly $$w_k=u_k \quad \mbox{ in} \quad B_{1-4\sigma}, \quad \quad w_k \le u_k \quad \mbox{in}  \quad B_{1-2 \sigma},$$
        and 
        $$w_k=v_k \quad \mbox{ in} \quad B_{1-\sigma}^c, \quad \quad w_k \le v_k \quad \mbox{in}  \quad B_{1-3 \sigma}^c.$$
     In $U$ we have
        \begin{equation*}
            |\nabla w_k| \leq C\sigma^{-1}|u-v|+|\nabla u_k|+|\nabla v_k|,
        \end{equation*}
        hence $$ J_{\eps_k}(w_k,0,U)\leq3 J_{\eps_k}(u_k,0,U)+3J_{\eps_k}(v_k,0,U)+o(1).$$

        {\it Step 2:} We split the inner annular region $B_{1-2\sigma} \setminus B_{1-3 \sigma}$ in rings of width $\eps=\eps_k$. Let
        $$r_i:=1-3\sigma+i \eps, \quad \quad \mbox{ with $1\leq i\leq N:=[\sigma/\eps]$,}$$ 
        and denote $\mathbb{A}_{i}=B_{r_{i}}\setminus B_{r_{i}-\eps}$. Let $\psi_{i}$ be supported in $B_{r_{i}}$ so that
        \begin{equation*}
            \psi_{i}\Big|_{B_{r_{i}-\eps}}=1,\quad|\nabla\psi_{i}|\leq C \eps^{-1} \chi_{\mathbb{A}_{i}}.
        \end{equation*}
        Let $\mathcal{G}_{i,k}$ be a candidates of $\mathcal{G}_k$ that interpolates $\mathcal E_k$ and $\mathcal F_k$ in $\mathbb{A}_i$:
        \begin{equation}
            \mathcal{G}_{i,k}:=\psi_i \mathcal{F}_k+(1-\psi_i)\mathcal{E}_k.
        \end{equation}
        Then in $\mathbb{A}_{i}$, 
        $$|\mathcal{G}_{i,k}-\mathcal{E}_k|\leq|\mathcal{F}_k-\mathcal{E}_k|\leq2,$$ 
        and
        \begin{equation*}
            |\nabla\mathcal{G}_{i,k}|\leq|\nabla\mathcal{E}_k|+|\nabla\mathcal{F}_k|+C \eps^{-1}|\mathcal{F}_k-\mathcal{E}_k|,
        \end{equation*}
        hence
        \begin{align*}
            J_{R}(0,\mathcal{G}_{i,R},\mathbb{A}_{i})\leq3J_{R}(0,\mathcal{E}_{R},\mathbb{A}_{i})+3J_{R}(0,\mathcal{F}_{R},\mathbb{A}_{i})+C_{1}\eps^{-1}\|\mathcal{F}_{R}-\mathcal{E}_{R}\|_{L^{1}(\mathbb{A}_{i})},
        \end{align*}
        with $C_{1}$ a constant depending on $\max|W'|$. We can choose an $1\leq i^{*}\leq N$, so that the last term
        \begin{equation*}
            C_{1}\eps^{-1}\|\mathcal{F}_k-\mathcal{E}_k\|_{L^{1}(\mathbb{A}_{i^{*}})}\leq C(\sigma) \|\mathcal{F}_k-\mathcal{E}_k\|_{L^{1}(B_{1-2\sigma} \setminus B_{1-3 \sigma})}.
        \end{equation*}
        For such an $i^{*}$, we define $\mathcal{G}_k=\mathcal{G}_{i^{*},k}$ and have
        \begin{align*}
            J_{\eps}(0,\mathcal{G}_k,U)=&J_{\eps}(0,\mathcal{F}_k, U\cap B_{r_{i}-\eps})+J_{\eps}(0,\mathcal{E}_k,U\setminus B_{r_{i}})\\
            & \quad +J_{\eps}(0,\mathcal{G}_{i^{*},\eps},\mathbb{A}_{i^{*}})\\
            \leq&3J_{\eps}(0,\mathcal{F}_k,U)+3J_{\eps}(0,\mathcal{E}_k,U)+o(1).
        \end{align*}

       It remains to check the pair $(w_k,\mathcal{G}_k)$ is admissible, meaning that $$\{w_k>0\} \subset \{\mathcal G_k=1\}.$$ This is obvious since by construction in $B_{1-2\sigma} \setminus B_{1-3 \sigma}$ we have
        $$w_k= \min\{u_k,v_k\} \quad \mbox{and} \quad  \{\mathcal E_k=1\} \cap \{\mathcal F_k=1\}  \subset \{G_k=1\}.$$    
\end{proof}

\begin{proof}[Proof of Theorem \ref{gamma convergence}]
    Since $v_k$ are uniformly bounded in $H^1(B_1)$, and $H(\mathcal E_k)$ are uniformly bounded in $BV(B_1)$ by \eqref{Hin}, the standard compactness results show that we can extract convergent subsequences 
    $$v_k \to v \quad \mbox{in} \quad L^2(B_1), \quad \quad \mathcal E_k \to \mathcal E \quad \mbox{in} \quad L^1(B_1).$$
    Using that
    $$ \int_{B_1} W(\mathcal E_k) dx \le M \eps_k \to 0,$$
    we conclude that $\mathcal E= \pm 1$ a.e. in $B_1$, hence
    $$ \mathcal E= \chi_{E^c} - \chi_E,$$
    for a Caccioppoli set $E$. Lemma \ref{one-phase holder} gives that the $v_k$'s have bounded $C^{1/2}$ norms locally, hence $v_k \to v$ in $C^\alpha_{loc}(B_1)$.
If $v(x)>0$ for some $x\in B_{1}$, then $v_k>0$ in some neighborhood of $x$ due to $C^{\alpha}$ convergence. As $(v_k,\mathcal{E}_k)\in\mathcal{A}(B_{1})$, we have $\mathcal{E}_k=1$ near $x$, hence also $\mathcal{E}=1$ near $x$ which proves $(v,E)\in\mathcal{A}_0(B_{1})$. 

It remains to show that $(v,E)$ minimizes $I$ in $B_1$. 
Fix $r<1$ and let $(\tilde v, \tilde E)$ be a minimizing pair for $I$ among all pairs which coincide with $(v,E)$ in $B_1 \setminus B_r$ and that are unrestricted in $B_r$. 

Notice that $\tilde v$ is continuous according to Lemma \ref{ACKS continuity} since the boundary data $v$ is H\"older continuous. Then we apply Theorem \ref{TGC} and obtain an approximating sequence $(u_k, \mathcal F_k) \in \mathcal A(B_1)$ such that
$$ J_{\eps_k}(u_k, \mathcal F_k, B_1) \to I( \tilde v, \tilde E,B_1).$$  
Using the uniform bound of the energies of $u_k$ and $v_k$ in $B_1$, given $\delta>0$ we can choose an annular region $U=B_{1-\sigma} \setminus B_{1-4 \sigma} \subset B_r^c,$ with $\sigma$ sufficiently small such that, after passing to a subsequence, we have
$$ J_{\eps_k}(v_k, \mathcal E_k, U)+J_{\eps_k}(u_k, \mathcal F_k, U) \le \delta,$$
for all large $k$. By Lemma \ref{gluing}, we find $(w_k, \mathcal G_k) \in  \mathcal A(B_1)$ which interpolate between the pairs $(v_k, \mathcal E_k)$ and $(u_k, \mathcal F_k)$ in the region $U$. Using the minimality of $v_k$ in $B_1$,
$$ J_{\eps_k}(v_k,\mathcal E_k,B_1) \le J_{\eps_k}(w_k,\mathcal G_k,B_1),$$
and the conclusion of Lemma \ref{gluing} we infer that
$$ J_{\eps_k}(v_k,\mathcal E_k,B_{1-4\sigma}) \le J_{\eps_k}(u_k,\mathcal E_k,B_{1-\sigma}) + 4\delta.$$
We let $k \to \infty$ and use the lower semicontinuity property to get
$$I(v,E,B_{1-4 \sigma}) \le I(\tilde v,\tilde E,B_1) + 4 \delta.$$
We let $\sigma \to 0$ and then $\delta \to 0$ to obtain
$$ I(v,E,B_1) \le I(\tilde v,\tilde E,B_1).$$

\end{proof}

\section{Unbounded solutions on both sides}

In this section we prove Theorem \ref{GM} part b). 

We assume throughout that $u$ is a global minimizer of $J$ which satisfies part ii) of Definition \ref{bcases}, that is, 
\begin{equation}\label{ccase}
 \{u <-1\} \ne \emptyset, \quad \quad \{ u>1 \} \ne \emptyset. 
 \end{equation}

Denote by
$$M(R):= \max_{B_R} |u|,$$
and, as in Theorem \ref{GM} part b), we assume that
\begin{equation}\label{mr}
M(R)=o(R^2) \quad \quad \mbox{as} \quad R \to \infty.
\end{equation}
Moreover, by Lemma \ref{l25}, \eqref{ccase} implies the existence of a small constant $\delta_0>0$ such that
\begin{equation}\label{mr2}
M(R) \ge \delta_0 R \quad \mbox{for all large $R$.}
\end{equation}

\subsection{Asymptotic flatness}

We first show that $u$ is asymptotically linear at infinity along a subsequence.

\begin{prop}\label{AL} Assume that $u$ is a global minimizer of $J$ and \eqref{ccase}-\eqref{mr} hold. There exist sequences of $R_k \to \infty$, $\eps_k \to 0$ and $a_k \ge \delta_0$ such that (up to a rotation)
$$ |u(x) - a_k x_n| \le \eps_k a_k R_k \quad \mbox{in} \quad B_{R_k}.$$
\end{prop}

\begin{rem}\label{remal} The lower bound $a_k \ge \delta_0$ is simply a consequence of \eqref{mr2}. Since $$\frac{U_a(t)}{t} \to a \quad \mbox{ as} \quad  t \to \pm \infty,$$ where $U_a$ denotes the one-dimensional solution (see Definition \ref{Ua}), the conclusion can be rewritten as
$$U_a(x_n - \eps R) \le u(x) \le U_a(x_n + \eps R),$$
for a sequence of $R_k \to \infty$, $\eps_k \to 0$, $a_k \ge \delta_0$.
\end{rem}

The strategy to prove Proposition \ref{AL} is to consider the rescaled functions
$$u_R(x):= \frac{1}{M(R)} \cdot u(Rx),$$
which are bounded by 1 in $B_1$, and show that along subsequences they converge to a harmonic function in $B_1$. Notice that $u_R$ minimizes the rescaled energy $J_R$ in $B_1$
$$J_R(v,B_1):= \int_{B_1} \frac 12 |\nabla v|^2 +  W_R(v)\, dx,$$
with potential
\begin{equation}\label{wr}
W_R(v):= \left( \frac{R}{M(R)}\right )^2 \cdot W(M(R) v).
\end{equation}
In view of \eqref{mr2}, the $W_R$'s are bounded in $L^\infty$ uniformly in $R$. We first prove an elementary lemma which gives the compactness in $C^\alpha_{loc}$ of minimizers of $J_R$.

\begin{lem}[H\"older estimate]\label{two phase holder}
Let $v$ be a minimizer in $B_{1}$ of an energy functional
    \begin{equation*}
       \int_{B_1}\frac 12 |\nabla v|^{2}+F(x,v)\,dx
    \end{equation*}
    with $\|F\|_{L^\infty}\leq1$ and $\|v\|_{L^{2}(B_{1})}\leq1$. Then 
    $$\|v\|_{H^{1}(B_{1/2})}\leq C, \quad \mbox{and} \quad \|v\|_{C^{\alpha}(B_{1/2})}\leq C ,$$
    for some $\alpha>0$.
\end{lem}
\begin{proof} The $H^1(B_{1/2})$ bound for $u$  follows from Caccioppoli's inequality which is obtain by comparing the energies of $v$ and $(1-\varphi) v$ in $B_1$, and using that $\|F\|_{L^\infty}\leq1$. Here $\varphi$ denotes a cutoff function which is $1$ in $B_{1/2}$ and vanishes outside $B_{3/4}$.

For the H\"older continuity of $v$, it suffices to show that for balls $B_r(x_0) \subset B_{3/4}$ we have
$$\int_{B_{r}(x_0)} |\nabla v|^2 \, dx \le C_0 r^{n+2\alpha -2}.$$
This follows from the standard Campanato iteration. Indeed, assume the desired bound holds in the ball $B_r(x_0)$, and then we need to show that it holds also in $B_{\rho r}(x_0)$ for some $\rho$ small universal. Let $\bar v$ denote the harmonic replacement of $v$ in $B_r(x_0)$. Then the minimality of $v$ implies
$$\int_{B_r(x_0)} |\nabla v - \nabla \bar v|^2 dx =  \int_{B_r} |\nabla v|^2 - |\nabla \bar v|^2 \, dx \le 4 |B_r|.$$
This together with the interior gradient estimate for $\bar v$,
$$ \int_{B_{\rho r}(x_0)}|\nabla \bar v|^2 dx \le C \rho ^n \int_{B_{r}(x_0)}|\nabla \bar v|^2 dx  \le C \rho ^n \int_{B_{ r}(x_0)}|\nabla v|^2 dx,$$
gives
$$ \int_{B_{\rho r}(x_0)}|\nabla v|^2 dx \le C \rho ^n \cdot C_0r^{n+2\alpha -2} + C r^n \le C_0(\rho r)^{n+2\alpha -2},$$
provided that $\rho$ is chosen small depending on $\alpha \in (0,1)$ and $n$, and $C_0$ large.

\end{proof}

\begin{proof} [Proof of Proposition \ref{AL}] Along any sequence $R_k \to \infty$ we can find a subsequence of rescalings $u_{R_k}$ which converges in $C^\alpha_{loc}(B_1)$ to a limiting function $\bar v \in H^1_{loc}(B_1)$:
$$v_k:= u_{R_k}, \quad v_k \to v_\infty \quad \mbox{in} \quad C^\alpha_{loc}(B_1).$$
The limiting harmonic function $v_\infty$ satisfies $v_{\infty}(0)=0$, $|v_\infty| \le 1$.

\smallskip

{\it Claim}: $v_\infty$ is harmonic in $B_1$.

\smallskip

Assume by contradiction that $v_\infty$ is not harmonic. Then we can find a ball $B_r(x_0) \subset B_1$ (say $x_0=0$ for simplicity of notation), such that the harmonic replacement of $v_\infty$, denoted by $\bar v$ is not identically zero, and
$$\int_{B_r} |\nabla \bar v|^2dx \le \int_{B_r} |\nabla v_\infty|^2 dx - \sigma,$$
for some small $\sigma>0$. From \eqref{wr} we see that the $W_{R_k}$'s are uniformly bounded and converge pointwise to 0 except at the origin. Since $\bar{v}$ is almost everywhere non-zero in $B_{r}$, by Lebesgue dominated convergence theorem we have
\begin{equation}\label{bsr}
 \int_{B_r}{W}_{R_{k}}(\bar{v})dx \to  0 \quad \mbox{as $k \to \infty$}.
 \end{equation}
We consider a competitor $w_k$ for $v_k$ in $B_{2r}$ obtained by interpolating between $\bar v$ and $v_k$ in a small neighborhood outside $B_r$,
$$w_k:=\varphi \cdot \bar{v}+(1-\varphi) \cdot v_{k},$$
where $\varphi$ is a cutoff function which is 1 in $B_{r_{1,k}}$ and $0$ in $B_{r_{2.k}}$. Here the radii $r_{1,k}$ and $r_{2,k}$ are chosen such that
$$\int_{B_{r_{2,k}} \setminus B_{r_{1,k}} } |\nabla \bar v|^2 + |\nabla v_k|^2 dx  \le \frac \sigma 4,$$
and
 $$r \le r_{1,k} \le r_{2,k} \le r+\eta, \quad \quad r_{2,k}-r_{1,k} \ge c(\sigma,\eta),$$ 
for some fixed $\eta>0$ small. Using that $v_k \to \bar v$ uniformly outside $B_r$, and that $|W_R| \le C$, as in Lemma \ref{gluing} it is easy to check that for all $k$ large
$$J_{R_k}(w_k,B_{r+\eta}) \le J_{R_k} (\bar v, B_{r+\eta}) + J_{R_k}(v_k,B_{r+\eta} \setminus B_r) + \frac{\sigma}{4} + C |B_{r+\eta} \setminus B_r|.$$
The minimality of $v_k$ then implies
$$J_{R_k}(v_k,B_r) \le J_{R_k} (\bar v, B_{r+\eta}) + \frac{\sigma}{4} + C |B_{r+\eta} \setminus B_r|.$$
We let $k \to \infty$, use \eqref{bsr}, and then let $\eta \to 0$ to obtain
$$ \int_{B_r} \frac 12 |\nabla v_\infty|^2 dx \le  \int_{B_{r}} \frac 12 |\nabla \bar v|^2 dx + \frac{\sigma}{4}.$$
This is a contradiction and the claim is proved.

\smallskip

 We distinguish two cases:

{\it Case 1:} There exists a limiting function sequence $R_k \to \infty$ for which the corresponding limiting function $v_\infty$ satisfies $\nabla v_{\infty}(0) \ne 0$. 
Assume that $\nabla v_\infty (0)=a e_n$ for some $a>0$. Since $v_k$ converges uniformly to $$v_\infty(x)= a x_n + O(|x|^2),$$
we obtain the conclusion of the proposition in balls of radii $\rho_k R_k$ with $\rho_k$ a sequence that decreases slowly to $0$.
 
\smallskip

{\it Case 2:} All limiting functions $v_\infty$ satisfy $\nabla v_\infty(0)=0$. We first show that 
\begin{equation}\label{32}
M(R) \ge R^{3/2} \quad \mbox{for all large $R$.} 
\end{equation} 
By compactness, all rescaled functions $u_R$ for large $R$ are well approximated by a limiting harmonic function $v_\infty$, and the Case 2 assumption gives
$$|v_\infty(x)| \le C |x|^2 \quad \quad \mbox{in $B_1$,}$$
with $C$ depending only on $n$. This shows that 
$$\|u_R\|_{L^\infty(B_\rho)} \le \rho^{7/4}, $$
with $\rho$ a small fixed universal constant, hence $$M(\rho R) \le \rho^{7/4} M(R),$$
for all large $R$'s, which implies the claim \eqref{32}.

In view of \eqref{32}, the rescaled potential $W_R$ (see \eqref{wr}) satisfies
$$\|W_R\|_{L^\infty} \le R^{-1}.$$
This means that $u_R$ is better approximated by a harmonic function. Indeed, let $v_R$ be the harmonic function in $B_1$ with boundary data $u_R$. The minimality of $u_R$ for the rescaled energy $J_R$ gives
$$ \int_{B_1}\frac 12 |\nabla (u_R-v_R)|^2 dx \le \int_{B_1}W_R(v_R)-W_R(u_R) dx \le R^{-1}.$$
Using the uniform H\"older continuity of $u_R$ and $v_R$ in $B_{1/2}$, we find
$$|u_R-v_R| \le R^{-\sigma} \quad \mbox{in} \quad B_{1/2},$$
for some $\sigma>0$ small, universal. From the $C^3$ estimates for $v_R$, and since $$u_R(0) = O(M(R)^{-1}),$$ we obtain
$$|u_R- p_R\cdot x - \frac 12 x^T A_R x| \le C (R^{-\sigma} +|x|^3) \quad \mbox{in} \quad B_{1/2},$$
for some $p_R \in \R^n$, $A_R \in \R^{n \times n}$ with $|p_R|, \|A_R\| \le C$. 

If $|p_{R}| \ge R^{-\sigma/4}$ for a sequence of $R=R_k \to \infty$, then in the ball $B_\rho$ of radius $\rho=R^{-\sigma/2}$ we have
$$|u_R-p_R \cdot x| \le C \rho^2 \le \varepsilon  \cdot |p_R| \rho , \quad \quad \varepsilon:=C R^{-\sigma/4} \to 0,$$
and this implies the desired conclusion for $u$ in balls of radius $\rho R=R^{1-\sigma/2}$.

Next, let's suppose that $|p_{R}| \le R^{-\sigma/4}$ for all large $R$'s, and show this contradicts assumption \eqref{mr}. The inequality above implies
$$|u_R - \frac 12 x^T A_R x| \le C (R^{-\sigma/4} +|x|^3) \quad \mbox{in} \quad B_{1/2}.$$

Now we set $\rho=R^{-\sigma/12}$ and have that
\begin{equation}\label{car}
|u_R-\frac 12 x^TA_R x| \le C \rho^3 \quad \mbox{in} \quad B_{2\rho}.
\end{equation}
If $\|A_R\| \le c$ with $c$ small, then $|u_R| \le \rho^2$ in $B_\rho$ which means 
\begin{equation}\label{car2}
\frac{M(r)}{r^2}  \le \frac{M(R)}{R^2}, \quad \quad \mbox{with} \quad r:= \rho R=R^{1-\sigma/12}.
\end{equation}
If $\|A_R\| \ge c$ then \eqref{car} implies
\begin{equation}\label{car3}
\frac{M(r)}{r^2} \le (1+ C\rho)\frac {M(2r)}{(2r)^2}.
\end{equation}
In conclusion for all large values of $r$ we have that either \eqref{car2} holds with $R=r^{1+\mu_1}$ for some constant $\mu_1>0$, or \eqref{car3} holds with $\rho=r^{-\mu_2}$ for some $\mu_2>0$. This property easily implies that $M(r)/r^2$ is bounded below along a sequence of $r=r_k \to \infty$, which contradicts our assumption \eqref{mr}.

\end{proof}

\subsection{Improvement of flatness}

Whenever we have a function $u$ that is well approximated by a one dimensional solution $U_a$ in $B_R$
\begin{equation}\label{flat4}
U_a(x_n - \eps R) \le u(x) \le U_a(x_n + \eps R), \quad \mbox{in $B_R$},
\end{equation}
we define its rescaling $\tilde u$ in $B_1$ as
$$ \tilde u(x):=  \frac {1}{\eps} \left[U_{a}^{-1}(u(R x)) -x_n\right], \quad \quad \quad x\in B_1.$$
Equivalently $\tilde u$ is defined by the formula
\begin{equation}\label{u(x)2}
u(x) = U_{a} \left(x_n + \tilde u( \frac {x}{R})\cdot \eps R \right), \quad \quad x \in B_{R}.
\end{equation}
In terms of the rescaling $\tilde u$, \eqref{flat4} is equivalent to
$$ |\tilde u| \le 1 \quad \mbox{in} \quad B_1.$$
If $a \ge \delta$ for some small parameter $\delta$, then $U_a$ is linear outside a compact interval whose length depends on $\delta$ and \eqref{flat4} guarantees that $\tilde u$ is harmonic in $B_1$ outside a small strip around $x_n=0$. Precisely, if we assume that the translation parameter $\eps R \ge \mu$ then we find
$$ \triangle \tilde u =0 \quad \mbox{in} \quad \{|x_n| \ge C(\delta,\mu) \, \eps \} \cap B_1.$$

We will show that $\tilde u$ is well approximated by a harmonic function in the whole domain $B_1$ as we let $\eps \to 0$. For this we use the explicit family of comparison functions constructed in Lemma \ref{l210}, and the behavior of their rescalings near $x_n=0$. 

\begin{lem}\label{phis} Fix $\delta$ and $\mu$ two small parameters. Let $a \ge \delta$, and let $P$ be a quadratic polynomial of the form, 
\begin{equation}\label{P}
P(x)= p + q \cdot x - \frac K 2 |x'|^2, \quad \quad K \ge \mu,
\end{equation}
with coefficients bounded by $\mu^{-1}$. 
There exists a subsolution $\Phi_P$ to \eqref{AC} which is an approximation of the one dimensional solution $U_a$ in $B_R$ such that its rescaling $\tilde \Phi_P$ defined in $B_1$ by \eqref{u(x)2} satisfies 
$$ \tilde \Phi_P = P(x) + O(x_n^2) + O(\eps),$$
provided that $$\eps R \ge \mu \quad \mbox{and} \quad  \eps \le \eps_0(\delta,\mu).$$ Here $\eps_0$ is sufficiently small depending on $\delta$, $\mu$, $n$ and $\|W\|_{L^\infty}$, and the constants in $O(\cdot)$ depend on the same quantities as well.
\end{lem}
     
   \begin{proof} After a translation in the $x'$ direction we may assume that $P$ is of the form
   $$P= p + qx_n - \frac K 2 |x'|^2,$$
   with $q \in \R$. Let $\Phi$ be the subsolution constructed in Lemma \ref{l210} with $\eps$ replaced by $\bar \eps=K \eps$ and $a$ replaced by $\bar a=a(1+ \eps q)$. After a translation of vector $(R/ \bar \eps -pR \eps) e_n$, the function $\Phi$ can be written 
   $$\Phi (x) =U_{\bar a} \circ \tau (d), \quad \tau(d)=d+ O(\frac{\eps}{R})d^2$$
   where $d$ is the signed distance to the sphere of radius $R/\bar \eps$ centered at $(R/ \bar \eps -p R \eps) e_n$, positive inside the ball and negative outside. Since by \eqref{usigma}
   $$U_a^{-1} \circ U_{\bar a}(t)=(1+\eps q) t + O(\eps q),$$
   we have
   $$ \Phi(x)=U_a \left ((1+\eps q) d +O(\frac{\eps}{R})d^2 + O(\eps)\right).$$
  Using that in $B_R$
   \begin{align*}
   d &= x_n + p R \eps - \frac {K \eps}{2R} |x'|^2 + R\cdot O(\eps^2) \\
   & = x_n + R \cdot O(\eps),
   \end{align*}
   we find
   $$U_a^{-1} \circ \Phi =  x_n + \left(p + q \frac{x_n}{R} - \frac K2 |\frac{x'}{R}|^2 + O\left((\frac{x_n}{R})^2\right) + O(\eps)\right)\cdot \eps R + O(\eps).$$
   Since $\eps R \ge \mu$, the last term $O(\eps)$ can be absorbed into $O(\eps) \cdot \eps R$ and the lemma is proved by recalling the definition of $\tilde \Phi$ from \eqref{u(x)2}.

   \end{proof}

Next we prove a version of the Harnack inequality for the rescaling $\tilde u$.

\begin{lem} [Harnack inequality]  \label{Hi} Fix $\delta,\mu>0$. There exist small constants $\eps_0$ depending on $\delta$, $\mu$, $n$ and $\|W\|_{L^\infty}$ such that if 
       \begin{equation*}
        U_{a}(x_{n})\leq u(x) \quad \mbox{ in }\quad B_{R}, \quad \mbox{and}  \quad U_{a}(\bar x_{n}+\sigma)\leq u(\bar x) \quad \mbox{at} \quad \bar x=\frac R 2 e_n,
         \end{equation*}
        with 
        \begin{equation*}\label{cond}
        a\ge \delta, \quad \quad \mu \le \sigma \le \eps_0R,
        \end{equation*}
   then 
    \begin{equation*}
        U_{a}(x_{n}+c \sigma )\leq u(x) \quad \mbox{ in }\quad B_{R/2},
    \end{equation*}
    for some constant $c$ depending only on $n$.
\end{lem}

As a consequence of the Lemma \ref{Hi} we have the following result.
\begin{cor}\label{chi} Assume that      
\begin{equation*}
        U_{a}(x_{n} + \beta_1)\leq u(x)\leq U_{a}(x_{n}+\beta_2)\quad \mbox{ in }\quad B_{R},
         \end{equation*}
   and $$a \ge \delta, \quad \mu \le \beta_2-\beta_1 \le \eps_0 R.$$ 
   Then 
    \begin{equation*}
        U_{a}(x_{n}+\bar \beta_{1})\leq u(x)\leq U_{a}(x_{n}+\bar \beta_{2})\quad \mbox{ in }\quad B_{R/2},
    \end{equation*}
    with $$\bar \beta_{2}- \bar \beta_{1}\le (1-c)(\beta_2-\beta_1).$$
    \end{cor}

If $|\beta_1| \ge \frac 3 4R$, it is simply a consequence of the classical Harnack inequality. Otherwise, after a translation, we end up in the situation of Lemma \ref{Hi} and apply its conclusion.
\begin{proof} [Proof of Lemma \ref{Hi}] We write $\sigma=\eps R$, and in terms of the rescaling $\tilde u$ defied in \eqref{u(x)2} the hypotheses can be rephrased as $$ \tilde u \ge 0 \quad \mbox{in} \quad B_1, \quad \quad \tilde u (\bar x_0) \ge 1, \quad \bar x_0:= e_n/2,$$
and we need to show that $\tilde u \ge c$ in say $B_{1/8}$. 

Recall that $\tilde u$ is harmonic in the set $x_n \ge C(\delta) \eps$, thus $u \ge c_1$ in $B_{1/4}(\bar x_0)$ by the classical Harnack inequality. 

It suffices to show that $\tilde u$ satisfies the comparison principle in $B_1 \setminus B_{1/4}(\bar x_0)$ with $c' \Gamma$, where $\Gamma$ is a slight modification of the fundamental solution with pole at $\bar x_0$ which vanishes on $\p B_{3/4}(\bar x_0)$,
$$ \Gamma:=\left(|x-\bar x_0|^{2-n} - (3/4)^{2-n} \right) + c_0 x_n^+,$$
with $c'$ and $c_0$ small constants that depend only on $n$. 
Indeed, since $\Gamma$ is harmonic away from $\{x_n=0\}$, it follows that $\tilde u- c' \Gamma$ cannot have an interior negative minimum in the region 
$$\{|x_n| \ge C(\delta,\mu) \eps\} \cap \left (B_1 \setminus B_{1/4}(\bar x_0)\right).$$
On the other hand if such a minimum occurs in the strip $|x_n| \ge C(\delta,\mu) \eps$ then $\tilde u$ can be touched by below at an interior point by a function $\tilde \Phi_P$ constructed in Lemma \ref{phis} for an appropriate polynomial $P$ and we reach a contradiction.  

\end{proof}

The Harnack inequality can be upgraded to an improvement of flatness result for solutions $u$.

\begin{prop}[Improvement of flatness]\label{impft}
    Fix $\delta>0$. There exists a constant $\eps_0$ small, depending on $\delta$, $n$ and $\|W\|_{L^\infty}$ such that if $u$ solves \eqref{AC}, $u(0)=0$, and
       \begin{equation}\label{flat1}
        U_{a}(x_{n}-\sigma)\leq u(x)\leq U_{a}(x_{n}+\sigma)\quad \mbox{ in }\quad B_{R},
         \end{equation}
        with $$a\ge \delta, \quad \delta \le \sigma \le \eps_0 R,$$
   then 
    \begin{equation}\label{flat2}
        U_{\bar a}(x \cdot \bar \nu - \bar \sigma ) \leq u(x)\leq U_{\bar a}(x \cdot \bar \nu + \bar \sigma)\quad \mbox{ in }\quad B_{\bar R},
    \end{equation}
    with $\bar \nu$ a unit direction and 
    $$\bar R = \rho R, \quad \quad \bar \sigma= \frac \rho 2 \sigma, \quad \quad |\bar \nu - e_n| + | \frac {\bar a}{ a} - 1|\leq C \frac \sigma R.$$
    Here $\rho$ small, $C$ large, are constants that depend only on $n$.
\end{prop}

Our main result Theorem \ref{GM} part b) follows easily from Proposition \ref{impft}.
\begin{proof}[Proof of Theorem \ref{GM} part b)] After a translation we may assume that $u(0)=0$. Fix $\delta \ll \delta_0$, and by Proposition \ref{AL} we know that \eqref{flat1} is satisfied for a sequence of large $R$'s, with $\sigma= \eps_0 R$. We start with this configuration and apply Proposition \ref{impft} iteratively, as long as the hypotheses are satisfied. We obtain \eqref{flat2} in balls of radius 
$$\bar R_k=\rho^k R, \quad \mbox{ with} \quad \bar \sigma_k=\eps_0 2^{-k} \bar R_k,$$
and we have to stop the iteration when $\bar \sigma_k \le \delta$ for the first time. Notice that the lower bound on $\bar a_k$ is always satisfied since
$$ | \frac {\bar a_{k+1}}{ \bar a_{k}} - 1|\leq C 2^{-k} \eps_0, \quad \quad \bar a_0 \ge \delta_0.$$
The last value of $\bar R_k$ in our iteration tends to infinity as we let the original $R \to \infty$. Thus, in the limit we obtain the inequalities in \eqref{flat2} are valid in the whole space $\R^n$ with $\bar \sigma = \delta$ and for some values of $\bar a$ and $\bar \nu$. Now we let $\delta \to 0$ and reach the desired conclusion.

\end{proof}

\begin{proof}[Proof of Proposition \ref{impft}] The proof is by compactness. Assume by contradiction that there exists sequences of $a_k$, $\sigma_k$, $R_k$, $u_k$ such that $u_k(0)=0$,
$$ a_k \ge \delta, \quad \sigma_k \ge \delta, \quad \eps_k:=\frac{\sigma_k}{R_k} \to 0,$$
and \eqref{flat1} is satisfied, but the conclusion does not hold for some fixed constants $\rho$, $C$ that will be specified later. 

Let $\tilde u_k$ be the rescaled functions given by \eqref{u(x)2} and then the hypotheses imply
$$ |\tilde u_k| \le 1 \quad \mbox{in} \quad B_1, \quad \quad \tilde u_k(0)=0.$$
{\it Claim:} The rescaled functions $\tilde u_k$ converge uniformly on compact sets of $B_1$ along a subsequence to a harmonic function $v$ with $v(0)=0$.

\

We iterate Corollary \ref{chi} as long as the hypotheses are satisfied and each iteration provides the diminish of oscillation of the $\tilde u_k$ in dyadic balls. Notice that the number of iterations permitted tends to infinity as we let $\eps_k \to 0$ since we may take the parameter $\mu \to 0$ as well. As a consequence the modulus of continuity of the $\tilde u_k$ converges to a uniform H\"older modulus of continuity as $k \to \infty$. The uniform convergence to a limiting continuous function $v$ along subsequences is then a consequence of the Arzela-Ascoli theorem. Clearly $v$ is harmonic in $B_1$ away from $\{x_n=0\}$ since the $\tilde u_k$ are harmonic outside a strip $|x_n| \le C(\delta) \eps_k$.

It remains to show that $v$ is harmonic on $\{x_n=0\}$ in the viscosity sense. For this, due to Hopf lemma, it suffices to prove that $v$ cannot be touched locally by below (or above) at a point on $x_n=0$ by a function of the type
$$ P+ \eta |x_n|, \quad \mbox{with $\eta> 0$ (or $\eta<0$)},$$
with $P$ a polynomial as in \eqref{P}. This is indeed the case since the $\tilde u_k$'s (hence $v$ as well) satisfy the comparison with the functions $\tilde \Phi_P$ constructed in Lemma \ref{phis}, and the claim is proved.

\

Since the limiting function $v$ is harmonic, there are constants $\rho$, $C$ depending only on $n$ such that
$$|v- \xi \cdot x| \le \frac{\rho}{4} \quad \mbox{in} \quad B_\rho.$$
Using that $U_a$ is increasing and that $\tilde u_k$ converges uniformly to $v$, we obtain that 
$$U_{a_k}(x_{n} + \eps_k \xi \cdot x  - \frac \rho 3  \cdot t_k) \leq u_k(x)\leq U_{a_k}(x_{n} + \eps_k \xi \cdot x  + \frac \rho 3  \cdot t_k) \quad \quad \mbox{in} \quad B_{\rho R_k}.$$
Denote by $$f_k:=e_n + \eps_k \cdot \xi, \quad \quad \nu_k:=f_k/|f_k|, \quad |f_k|=1+O(\eps_k),$$
and recall that by \eqref{usigma} if $a \ge \delta$ and $\gamma \in ( \frac 12,2)$,
$$ U_{\gamma a} (s - C (\gamma-1)) \le U_a(\gamma s) \le U_{\gamma a} (s + C (\gamma-1)) \quad \quad \quad \forall \quad s \in \R,$$
with $C$ a constant that depends on $\delta$. Then it follows that
$$ U_{\bar a_k} (x \cdot \nu_k - \frac \rho 2  \cdot \sigma_k ) \leq u_k(x) \le U_{\bar a_k} (x \cdot \nu_k + \frac \rho 2  \cdot \sigma_k ) \quad \quad \mbox{in} \quad B_{\rho R_k},$$
with $\bar a_k=|f_k| a_k$. This means that $u_k$ satisfies the conclusion \eqref{flat2} and we reached a contradiction.

\end{proof}


\section{Solutions with graphical level sets}

In this last section we use the results of Theorem \ref{GM} and prove a similar result in one dimension higher for critical points of $J$ that are monotone in one direction. Precisely, we consider monotone solutions of 
\begin{equation}\label{ACL}
\triangle u= W'(u)
\end{equation}
 with graphical $0$ level set i.e.
 \begin{equation}\label{ACL5}
 u_{x_n}>0, \quad \quad \mbox{$\{u=0\}$ is a graph over $\R^{n-1}$ in the $x_n$ direction.}
 \end{equation}
We recall from Section 2 that the potential $ W: \mathbb R \to [0, \infty)$ satisfies the following hypotheses:

a) $W=0$ outside the interval $[-1,1]$,

b) in the interval $[-1,1]$, $W$ is a $C^2$ function and
$$ W(\pm 1)=0, \quad W' (\pm 1)=0, \quad W''(\pm 1) >0,$$
$$ W'>0 \quad \mbox{in $(-1,0)$,} \quad W'<0 \quad \mbox{in $(0,1)$,} \quad W''(0)<0.$$ 
 
  \begin{thm}\label{Main3}
 Let $u$ be a global solution to \eqref{ACL} that satisfies \eqref{ACL5}. If $u =o(|x|^2)$ as $|x| \to \infty$, then $u$ is one-dimensional if $n \le 8$.
 \end{thm} 

We show the hypotheses of Theorem \ref{Main3} imply in fact that $u$ is a global minimizer for the energy $J$. Towards this aim we define $\bar u$ as the limit of $u$ at infinity,
$$ \bar u (x')= \lim_{x_n \to \infty} u(x',x_n) \quad \in (0, \infty].$$
\begin{lem} 
$\bar u \equiv \lambda^+$ for some $\lambda^+ \in [1,\infty]$.
\end{lem}
\begin{proof}
We show first that either $\bar u\equiv \infty$ or $\bar u < \infty$ in whole $\R^n$. Assume first that $\bar u$ takes the value $\infty$ at some point, say $\bar u(0)= \infty$. Fix $R>0$, and for each large $k \in \mathbb N$, we pick a point $x_k=(0,t_k)$ with $t_k$ sufficiently large such that $u(x_k) \ge k$, and $B_R(x_k)$  is included in $\{u>0\}$. Since $(u-1)^+ \ge 0$ is subharmonic and $u\ge 0$ is superharmonic in $B_R(x_k)$, we find from the mean value inequalities that $u \ge c_n u(x_k)$ in $B_{R/2}(x_k)$ with $c_n$ a constant that depends on $n$. This implies that $\bar u= \infty$ in $B_{R/2}$ and, since $R$ is arbitrary, we find that $\bar u\equiv \infty$.

Next we focus on the case $\bar u < \infty$ in $\R^n$. The trivial extension of $\bar u$ in the $x_n$ direction is locally the uniform limit of translations of the solution $u$, and this implies that $\bar u$ solves the same equation
\begin{equation}\label{ACL2}
\triangle \bar u=W'(\bar u) \quad \mbox{in} \quad \R^{n-1}.
\end{equation}
We claim that $\inf \bar u >0$. For this we construct an explicit subsolution $w=\delta \phi$ in a ball $B_R$, where $\phi$ denotes the first eigenvalue of the Laplacian in $B_R$. Since $W'(0)=0$, $W''(0) <0$, we can choose $R$ large and $\delta$ small such 
$$\triangle w=-\lambda_R w > W'(w), \quad \mbox{and $w < \bar u$ in $B_R$.}$$ 
By the maximum principle, as we move continuously the graphs of the translations of $w$, they always must remain below the graph of $\bar u$, and we find
$$\inf \bar u \ge \max_{B_R} w >0,$$
which proves the claim. 

On the other hand, the constant function $\inf \bar u$ is obtained as the infimum over translates of $\bar u$, hence it is a supersolution of the equation \eqref{ACL2}. Then
$$0 \le W'(\inf \bar u),$$
which gives $\inf \bar u \ge 1$. This means that $\bar u$ is an entire harmonic function that is bounded below, and the conclusion follows.
\end{proof}

\begin{proof}[Proof of Theorem \ref{Main3}] We show first that $u$ is a global minimizer by proving that it is the unique solution to \eqref{ACL2} with its own boundary data in a large ball $B_R$. Indeed, the maximum principle gives that the graph of $v$ is included in the region
$$\R^{n} \times (\lambda^-,\lambda^+) \subset \R^{n+1},$$ 
where $\lambda ^\pm$ denote the limits of $u$ at $\pm \infty$ obtained in the previous lemma. On the other hand, due to the monotonicity of $u$ in the $x_n$ direction, this region is foliated by translations of the graph of $u$. The maximum principle applied between $v$ and the leaves of the foliation imply that $v=u$. 

Now the conclusion follows from Theorem \ref{GM}. The gain of an extra dimension is due to the monotonicity assumption. This is because the restriction on $n$ is used only in the blow-down analysis performed in the cases when either $\lambda^\pm=\pm 1$ or when $u$ is bounded on one-side. Then the rescaled sets $\eps_k \{u>0\}$ converge to a global set $E$ which minimizes perimeter which must be a half-space if the dimension $n \le 7$. However, if $u$ is monotone in the $x_n$ direction, the set $E$ is an epigraph and then the dimension can be upgraded to $n\le8$.

\end{proof}

\end{document}